\pdfoutput=1 

\documentclass[12pt,a4paper]{amsart}
\usepackage[T1]{fontenc}
\usepackage[top=30mm, bottom=30mm, left=30mm, right=30mm]{geometry}    
\usepackage{amssymb}
\usepackage{mathtools}
\usepackage{verbatim,enumerate}
\usepackage{xcolor}
\usepackage[looser]{newpxtext}
\usepackage[smallerops]{newpxmath}

\usepackage{fancyhdr}
\makeatletter
\def\@@and{\MakeLowercase{and}}
\makeatother

\usepackage[hidelinks]{hyperref}

\theoremstyle{definition}
\newtheorem{defn}{Definition}[section]
\newtheorem{exam}[defn]{Example}

\newtheorem{rem}[defn]{Remark}

\theoremstyle{plain}
\newtheorem{thm}[defn]{Theorem}
\newtheorem{lem}[defn]{Lemma}
\newtheorem{prop}[defn]{Proposition}
\newtheorem{coro}[defn]{Corollary}

\newcommand{\eps}{\varepsilon}
\newcommand{\bbz}{\mathbb{Z}}
\newcommand{\bbn}{\mathbb{N}}
\newcommand{\bbk}{\mathbb{K}}
\newcommand{\bbr}{\mathbb{R}}
\newcommand{\bbc}{\mathbb{C}}
\newcommand{\bbs}{\mathbb{S}}
\newcommand{\vecz}{\mathbf{0}}
\DeclareMathOperator{\sspan}{span}
\DeclareMathOperator{\diam}{diam}
\DeclareMathOperator{\card}{\mathrm{card}}
\DeclareMathOperator{\udens}{\overline{\mathrm{dens}}}
\DeclareMathOperator{\dens}{\mathrm{dens}}
\DeclareMathOperator{\ldens}{\underline{\mathrm{dens}}}
\DeclareMathOperator{\eend}{End}
\DeclareMathOperator{\asym}{Asym}
\DeclareMathOperator{\prox}{Prox}
\DeclareMathOperator{\masym}{MAsym}
\DeclareMathOperator{\mprox}{MProx}
\DeclareMathOperator{\dasym}{DAsym}
\DeclareMathOperator{\dprox}{DProx}

\title[C\MakeLowercase{haos for endomorphisms and linear operators}] 
{C\MakeLowercase{haos for endomorphisms of completely metrizable groups and linear operators on} F\MakeLowercase{r\'echet spaces}}

\author[Z. J\MakeLowercase{iang} ]{Z\MakeLowercase{hen}  J\MakeLowercase{iang}}

\address[Z. Jiang]{Department of Mathematics,
	Shantou University, Shantou, 515821, Guangdong, China}
\email{jiangzhen17@outlook.com}
\urladdr{https://orcid.org/0000-0001-8013-9317}

\author[J. L\MakeLowercase{i}]{J\MakeLowercase{ian} L\MakeLowercase{i}}
\address[J. Li]{Department of Mathematics,
	Shantou University, Shantou, 515821, Guangdong, China}
\email{lijian09@mail.ustc.edu.cn}
\urladdr{https://orcid.org/0000-0002-8724-3050}

	

\date{\today}

\begin{document}

\begin{abstract}
    Using some techniques from topological dynamics, we give a uniform treatment of Li-Yorke chaos, mean Li-Yorke chaos and distributional chaos for continuous endomorphisms of completely metrizable groups, and characterize three kinds of chaos (resp.\@ extreme chaos) in terms of the existence of the so-called semi-irregular points (resp.\@ irregular points).
    We exhibit some examples of inner automorphisms of Polish groups to illustrate the results. We also apply our results to the chaos theory of continuous linear operators on Fr\'echet spaces, which improves some results in the literature.
\end{abstract}

\maketitle

\tableofcontents

\section{Introduction} \label{sec:Intro}

By a dynamical system, we mean a pair $(X,T)$, where $(X,d)$ is a metric space and $T\colon X\to X$ is a continuous map.
In the seminal paper~\cite{LY1975} Li and Yorke  initiated the study of chaos theory in dynamical systems.
Following~\cite{LY1975}, we say that a dynamical system $(X,T)$ is Li-Yorke chaotic if there exists an uncountable subset $K$ of $X$ such that for any two distinct points $x,y\in K$, one has
\[
    \liminf_{n\to\infty}d(T^nx,T^ny)=0\text{ and }\limsup_{n\to\infty}d(T^nx,T^ny)>0.
\]
Li and Yorke proved in~\cite{LY1975} that if a continuous interval map $T\colon [0,1]\to [0,1]$ has a periodic point of period $3$, then it is  Li-Yorke chaotic.
Since then, various alternatives, but closely related definitions of chaos,  have been proposed to capture the complexity of a dynamical system from different perspectives.

In order to characterize interval maps with positive topological entropy, Schweizer and Smital \cite{SS1994} introduced the notion of distributional chaos as a natural strengthening of Li-Yorke chaos.
Later in~\cite{SS04} and \cite{BSS05}  the authors further introduced three types of distributional chaos. In this paper, we will only focus on distributional chaos of types $1$ and $2$.
We say that a dynamical system  $(X,T)$ is distributionally chaotic of type $1$ (or type $2$) if there exists an uncountable subset $K$ of $X$ such that for any two distinct points $x,y\in X$, one has that for every $\eps>0$,
the upper density of the set $\{n\in\bbn\colon d(T^nx,T^ny)< \eps\}$ is one and there exists some $\delta>0$ such that the upper density of the set $\{n\in\bbn\colon d(T^nx,T^ny)>\delta\}$ is one (or positive).

It is observed by Downarowicz in~\cite{D2014} that for dynamical systems on a bounded metric space, distributional chaos of type $2$ is equivalent to another generalization of Li-Yorke chaos, namely mean Li-Yorke chaos, which was coined in~\cite{HLY2014}.
We say that  a dynamical system $(X,T)$ is mean Li-Yorke chaotic if there exists an uncountable subset $K$ of $X$ such that for any two distinct points $x,y\in K$, one has \[
    \liminf_{n\to\infty}\frac{1}{n}\sum_{i=1}^n d(T^ix,T^iy)=0\text{ and }\limsup_{n\to\infty}\frac{1}{n}\sum_{i=1}^n d(T^ix,T^iy)>0.
\]

For the study of chaos theory in topological dynamics, there are plenty of results about the existence of the different types of chaos and their relationships with each other.
We refer the reader to the survey \cite{LY2016} and references therein. 
Recently, a lot of attention has been paid to the study of chaotic phenomena in linear dynamics, that is, when $X$ is a topological vector space (usually a Banach or Fr\'echet space) and $T\colon X\to X$ is a continuous linear operator.  
We refer the reader to the books~\cite{BM2009, GP2011} for a thorough account of linear dynamics.
 
In a series of papers~\cite{BBMP2011, BBMP2013, BBMP2015, BBP2020, BBPW2018}, various chaotic properties related to Li-Yorke chaos were carefully considered, and many very interesting results were found in linear dynamics. For instance, in \cite{BBMP2011}, Berm\'udez et al.\ connected Li-Yorke chaotic operators on a Banach space with the existence of irregular points which is a topic discussed in operator theory. Later in~\cite{BBMP2015}, Bernardes et al.\ extended this result to Fr\'echet spaces. They utilized the linear structure and the local convexity to construct the uncountable scrambled set, and it is natural to ask if these ingredients are necessary to produce the Li-Yorke chaotic phenomena.

Every topological vector space is also a topological group. In~\cite{C2001}, Chan generalized the hypercyclicity criterion in linear dynamics to continuous homomorphisms on a separable completely metrizable topological semigroup.
In~\cite{M2009}, Moothathu studied the weak mixing and mixing properties of a continuous endomorphism of a Hausdorff topological group.
In \cite{CW2019}, Chan and Walmsley generalized 
the well-known universality criterion from linear dynamics to the setting of continuous homomorphisms acting on a separable, completely metrizable topological semigroup. 
In \cite{BP2022}, Burke and Pinheiro obtained some necessary and sufficient conditions for a continuous endomorphism on a Polish group to be weakly mixing.
In fact, the first chapter of the textbook \cite{B1974} on functional analysis and operator theory is about topological groups. 
It is then natural to investigate the dynamics of continuous endomorphisms of general topological groups and, since both topological dynamics and linear dynamics have plentiful results in chaos theory, it seems very interesting to connect these two important research areas. 

Based on the above reasons, 
influenced by the seminal work of Huang and Ye in~\cite{HY2002}, we take advantage of the techniques from topological dynamics to study the chaos theory in linear dynamics.
After distinguishing what is exactly needed in the linear case, we will show that many results also hold in the more general setting of continuous endomorphisms of completely metrizable groups.
We will give a uniform treatment of Li-Yorke chaos, mean Li-Yorke chaos and distributional chaos for continuous endomorphisms of completely metrizable groups and, as a consequence, these results will hold for continuous linear operators on F-spaces (completely metrizable topological vector spaces).
We will exhibit some examples of inner automorphisms of the Polish group of all permutations of $\bbz$ and the Polish group of all increasing homeomorphisms of $\bbr$ to illustrate the results.
We will also apply the general results on continuous endomorphisms of completely metrizable groups  to continuous linear operators on Fr\'echet and Banach spaces,
improving some results on the characterizations of chaos via irregular vectors in an abstract and concise way.
The paper is organized as follows.

In Section~\ref{sec:Pre} we recall some notation and preliminaries on topological spaces, topological groups, Fr\'echet spaces and dynamical systems.

In Section~\ref{sec:LYC}, we focus our attention on Li-Yorke chaos. For linear operators on Banach spaces, Berm\'{u}dez et al.~\cite{BBMP2011} gave a characterization of Li-Yorke chaos in terms of irregular vectors, introduced by Beauzamy in~\cite{B1988}. In~\cite{BBMP2015} the authors extended this interesting result from Banach to Fr\'echet spaces and systematically studied (dense) Li-Yorke chaos. In this section, we continue to extend the results to the more general situation of continuous endomorphisms acting on completely metrizable groups.
To do this, we first study equicontinuity, sensitivity, and Li-Yorke chaos in topological dynamics, and prove the dichotomy of equicontinuity and sensitivity for continuous endomorphisms of completely metrizable groups. Using the group structure and properties of sensitivity, we characterize Li-Yorke chaos (Li-Yorke extreme chaos) in terms of the existence of semi-irregular points (irregular points) for continuous endomorphisms of completely metrizable groups.
As an application, we characterize Li-Yorke chaos and Li-Yorke extreme chaos for continuous linear operators on Fr\'echet spaces. 

In Section~\ref{sec:MLYC}, we bring our attention to mean Li-Yorke chaos. In~\cite{BBP2020}, Bernardes et al.\ systematically studied (dense) mean Li-Yorke chaos for linear operators on Banach spaces.
In Subsection~\ref{subsec:Meq-Msen-MLYC}, we study mean equicontinuity, mean sensitivity and mean Li-Yorke chaos in topological dynamics. 
In Subsections~\ref{subsec:MLYC-group} and~\ref{subsec:MLYeC-group}, we prove the dichotomies of mean equicontinuity and mean sensitivity, and give a characterization of (dense) mean Li-Yorke (extreme) chaos for continuous endomorphisms of completely metrizable groups. 
In Subsection~\ref{subsec:MLYC-Frechet}, by using the linear structure and the results in Subsections ~\ref{subsec:MLYC-group} and~\ref{subsec:MLYeC-group}, we extended the main results in ~\cite{BBP2020} concerning (dense) mean Li-Yorke chaos to the Fr\'echet space setting. 

In Section~\ref{sec:DC}, we consider the case of distributional chaos. 
In~\cite{BBMP2013}, Bernardes et al.\ systematically studied (dense) distributional Li-Yorke chaos of type 1 for linear operators on Fr\'echet spaces. In~\cite{BBPW2018}, Bernardes et al.\ further studied (dense) distributional Li-Yorke chaos of type 1 and 2 for linear operators on Banach spaces. 
In Subsection~\ref{subsec:MLS-DC}, we study mean-L-stability,  mean-L-unstability and distributional Li-Yorke (extreme) chaos of type 1 and 2 in the context of topological dynamics. 
In Subsections~\ref{subsec:DC-group} and ~\ref{subsec:DeC-group}, we prove the dichotomies of mean-L-stability and mean-L-unstability, and give a characterization of distributional Li-Yorke (extreme) chaos of type 1 and 2 for continuous endomorphisms of completely metrizable groups, extending some results in the literature.
Finally, in Subsection 5.4, we use the linear structure and the results from Subsections 5.2 and 5.3 to give a more complete characterization, for dense distributional chaos of type 1 and type 2, than those given in \cite[Theorem 15]{BBMP2013} and \cite[Theorem 33]{BBPW2018} by Bernardes et al.\ in the context of linear operators acting on Fréchet spaces. Moreover, in Subsection 5.5 we explore the particular case in which the underlying space is Banach, closing the paper.

\section{Preliminaries}\label{sec:Pre}

Let $\bbn$ denote the set of all positive integers.
For a finite subset $A$ of $\bbn$, $\card(A)$ denotes the number of elements in $A$.
For a subset $A$ of $\bbn$, the \emph{upper density} of $A$ is defined by
\[
    \udens(A)=\limsup_{n\to\infty} \frac{\card(A\cap[1,n])}{n},
\]
and the \emph{lower density} of $A$ is defined by
\[
    \ldens(A)=\liminf_{n\to\infty} \frac{\card(A\cap[1,n])}{n}.
\]
We say that $A$ has \emph{density} $\dens(A)$  if $\udens(A)=\ldens(A)$, in which case $\dens(A)$ is equal to this common value.

The following result is a standard exercise in analysis.
\begin{lem}\label{lem:density-limit}
    Let $(a_n)_n$ be a sequence in $\bbr$. Then
    \begin{enumerate}
        \item there exists a subset $A$ of $\bbn$ with $\udens(A)=1$ such that $\lim\limits_{A\ni n\to\infty}|a_n|=0$ if and only if  for every $\eps>0$, $ \udens(\{ n\in\bbn\colon |a_n|< \eps\})=1$;
        \item there exists a subset $A$ of $\bbn$ with $\dens(A)=1$ such that $\lim\limits_{A\ni n\to\infty}|a_n|=0$ if and only if  for every $\eps>0$, $\dens(\{ n\in\bbn\colon |a_n|< \eps\})=1$;
        \item there exists a subset $A$ of $\bbn$ with $\udens(A)=1$ such that $\lim\limits_{A\ni n\to\infty}|a_n|=\infty$ if and only if  for every $M>0$, $ \udens(\{ n\in\bbn\colon |a_n|>M\})=1$.
    \end{enumerate}

\end{lem}

\subsection{Topological spaces}
Let $X$ be a topological space.
A subset of $X$ is called a \emph{$G_\delta$ set} if it can be expressed as the countable intersection of open subsets of $X$,
a \emph{Cantor set} if it is homeomorphic to the Cantor ternary set in the unit interval,
and a \emph{$\sigma$-Cantor set} if it can be expressed as the countable union of Cantor subsets of $X$.

We say that a topological space $X$ is \emph{completely metrizable} if there exists a metric $d$ on $X$ such that $(X, d)$ is a complete metric space and $d$ induces the topology of $X$.
The well-known Baire category theorem states that for every completely metrizable space, the intersection of a sequence of open dense sets is still dense.
In a completely metrizable space $X$, a subset of $X$ is called \emph{residual} if it contains a dense $G_\delta$ subset of $X$.

The following classical result is due to Mycielski \cite{M1964}, which is an important topological tool for our purposes.
\begin{thm}[Mycielski]\label{Mycielski}
    Let $X$ be a completely metrizable space without isolated points and $R$ be a dense $G_{\delta}$ subset of $X\times X$.
    Then for every sequence $(O_j)_j$ of nonempty open subsets of $X$, there exists a sequence $(K_j)_j$ of Cantor sets with $K_j\subset O_j$ such that
    \[
        \bigcup_{j=1}^\infty   K_j\times  \bigcup_{j=1}^\infty K_j\subset R\cup \Delta_X,
    \]
    where $\Delta_X=\{(x_1,x_2)\in X\times X\colon  x_1=x_2\}$.
    In addition, if $X$ is separable then we can require that the $\sigma$-Cantor set $\bigcup_{j=1}^\infty K_j$ is dense in $X$.
\end{thm}

\subsection{Topological groups}
Let $(G,\cdot)$ be a group and denote by $e$  the identity of the group $G$. We say that $G$ is a \emph{topological group}
if it is a group with a topology such that the group operations are continuous, a \emph{(completely) metrizable group} if it is a topological group with a (completely) metrizable topology, and a \emph{Polish group} if it is a separable completely metrizable group.

By the classical Birkhoff-Kakutani theorem (see e.g.~\cite[Theorem 6.3]{B1974}), every metrizable topological group $G$
admits a left-invariant compatible metric, that is, a metric $d$ inducing the group's topology and such that
\[
    d(zx,zy)=d(x,y)
\]
for all $x,y,z\in G$. 
It is worth mentioning that there exist natural examples of completely metrizable groups such that every left-invariant compatible metric is not complete, see e.g. \cite{M2011}. In this paper, whenever we consider a metrizable topological group $G$,
we fix a left-invariant compatible metric $d$ on $G$.
Denote by $\eend(G)$ the set of all continuous endomorphisms of $G$.

\subsection{Topological vector spaces}
Let $X$ be a topological vector space over a field $\bbk$ (standing for $\bbr$ or $\bbc$). Denote by $\vecz$ the zero vector in $X$.
We say that $X$ is an \emph{F-space} if it is a  completely metrizable topological vector space, and a \emph{Fr\'echet space} if  it is a locally convex F-space.

 In this paper, given a Fréchet space $X$ we will always consider a separating increasing sequence $(p_j)_j$ of seminorms endowing the topology of X. In this case we have that $X$ is complete in the Fr\'echet-space metric given by
\[
    d(x,y):=\sum_{j=1}^\infty \frac{1}{2^j} \min(1,p_j(x-y)),\quad x,y\in X.
\]
As $d$ is a bounded metric on $X$, the boundness with respect to $d$ makes no sense. We need the following boundness with respect to the linear structure of $X$.
We say that a subset $Y$ of $X$ is \emph{bounded}
if every neighborhood $V$ of $\vecz$ in $X$ corresponds a number $s>0$ such that  $Y\subset tV$ for every $t>s$.
It is easy to see that $Y$ is bounded if and only if  for any $j\geq 1$, $\sup_{x\in Y}p_j(x)<\infty$,
and $Y$ is unbounded if and only if there exists $m\in\bbn$
such that $\sup_{x\in Y} p_m(x)=\infty$.
Denote by $L(X)$ the set of all continuous linear operators on $X$.

Let $T\in L(X)$. It is clear that, with the additive operation of the vector space, $X$ is an abelian completely metrizable group, the metric $d$ is translation-invariant and $T$ is a continuous endomorphism of $X$.
So we can apply the results of continuous endomorphisms of completely metrizable groups to continuous linear operators on Fr\'echet spaces.

\subsection{Dynamical systems}

Recall that a \emph{dynamical system} is a pair $(X, T)$, where $(X,d)$ is a metric space and $T\colon X\to X$ is a continuous map. For a point $x\in X$, the \emph{orbit} of $x$ is the set $\{T^nx\colon n\geq 0\}$.
A subset $Y$ of $X$ is called \emph{$T$-invariant} if $TY\subset Y$.
Clearly, the orbit of $x$ is $T$-invariant, and the closure of a $T$-invariant set is also $T$-invariant.

When $X$ is a compact metric space, we say that $(X, T)$ is a \emph{compact dynamical system}. Usually, the study of chaos theory in topological dynamics deals with compact dynamical systems, see the survey~\cite{LY2016} and references therein.
On the other hand, in many situations, the results and techniques also hold for dynamical systems on completely metrizable spaces.
In Subsections \ref{subsec:Eq-sen-LYC},~\ref{subsec:Meq-Msen-MLYC} and~\ref{subsec:MLS-DC}, we study various versions of equicontinuity, sensitivity and chaos for dynamical systems on completely metrizable spaces.
Even though those results must be known for specialists in topological dynamics, we will provide some proofs for the sake of completeness.

\section{Li-Yorke chaos}\label{sec:LYC}

\subsection{Equicontinuity, sensitivity and Li-Yorke chaos in topological dynamics}\label{subsec:Eq-sen-LYC}
Let $(X,T)$ be a dynamical system with a metric $d$ on $X$.
A pair $(x,y)\in X\times X$ is called \emph{asymptotic} if
\[
    \lim_{n\to \infty}d(T^nx,T^ny)=0,
\]
and \emph{proximal} if
\[
    \liminf_{n\to \infty}d(T^nx,T^ny)=0.
\]
The \emph{asymptotic relation} and the \emph{proximal relation} of $(X,T)$, denoted by $\asym(T)$ and $\prox(T)$, are the set of all asymptotic pairs and proximal pairs respectively.
For any $x\in X$, the \emph{asymptotic cell} and the  \emph{proximal cell} of $x$ are defined by
\[
    \asym(T,x):=\{y\in X\colon (x,y)\in \asym(T)\}
\]
and
\[
    \prox(T,x):=\{y\in X\colon (x,y)\in \prox(T)\},
\]
respectively. 

To apply Theorem~\ref{Mycielski}, 
the $G_\delta$ property of the desired set will be critical. 
The following three lemmas are well-known, see e.g.~\cite{HY2002,LY2016}.
Here we provide the proofs for the sake of completeness.
 
\begin{lem}\label{lem:dyn-prox-delta}
    Let $(X,T)$ be a dynamical system and $\delta>0$.
    Then
    \begin{enumerate}
        \item for every $x\in X$, the proximal cell of $x$ is a $G_\delta$ subset of $X$;
        \item the proximal relation of $(X,T)$ is a $G_\delta$ subset of $X\times X$;
        \item for every $x\in X$, the set
              \[
                  \Bigl\{y\in X\colon \limsup_{n\to\infty}d(T^nx,T^ny)\geq \delta\Bigr\}
              \]  is a $G_\delta$ subset of $ X$;
        \item the set
              \[
                  \Bigl\{(x,y)\in X\times X\colon \limsup_{n\to\infty}d(T^nx,T^ny)\geq \delta\Bigr\}
              \]
              is a  $G_\delta$ subset of $X\times X$.
    \end{enumerate}
\end{lem}
\begin{proof}
    (1). Since
    \[
        \prox(T,x)=\bigcap_{n=1}^{\infty}\Bigl\{y\in X \colon \exists k>n \text{ s.t. } d(T^kx,T^ky)<\tfrac{1}{n}\Bigr\},
    \]
    by the continuity of $T$, the proximal cell $\prox(T,x)$ is a $G_{\delta}$ subset of $X$.

    (3). Note that
    \begin{align*}
        \Bigl\{y\in X\colon \limsup_{n\to\infty}d(T^nx,T^ny)\geq \delta\Bigr\}
        = \bigcap_{n=1}^{\infty}\Bigl\{y\in X\colon\exists k>n \text{ s.t. } d(T^k x,T^k y)>\delta-\tfrac{1}{n}\Bigr\}.
    \end{align*}
    By the continuity of $T$, the set is a $G_\delta$ subset of $X$.

    The proofs of (2) and (4) are similar to those of (1) and (3).
\end{proof}

We say that a dynamical system $(X,T)$ is \emph{equicontinuous} if for any $\eps>0$, there exists some $\delta>0$ such that $d(T^nx,T^ny)<\eps$ for all $x,y\in X$ with $d(x,y)<\delta$ and $n\in\bbn$.

\begin{lem}\label{lem:eq-prox-asym}
    Let $(X, T)$ be a dynamical system.
    If $(X, T)$ is equicontinuous, then every proximal pair is asymptotic.
\end{lem}
\begin{proof}
    Let $(x_0,y_0)\in X\times X$ be a proximal pair.
    Since $(X,T)$ is equicontinuous, for any $\eps>0$ there exists
    $\delta>0$ such that for any $x,y\in X$ with $d(x,y)<\delta$,
    $d(T^nx,T^nx)<\eps$ for all $n\in\bbn$.
    Since $(x_0,y_0)$ is proximal, there exists $k\in \bbn$ such that
    $d(T^kx_0,T^ky_0)<\delta$. Then $d(T^{k+n}x_0,T^{k+n}y_0)<\eps$ for all $n\in\bbn$ and
    thus $\limsup_{n\to\infty}d(T^n x_0,T^ny_0)\leq \eps$.
    By the arbitrariness of $\eps>0$, $(x_0,y_0)$ is asymptotic.
\end{proof}

We say that a dynamical system $(X,T)$ is \emph{sensitive dependent on initial conditions} (\emph{sensitive} for short) if there exists $\delta>0$ such that for every $x\in X$ and $\eps>0$, there exists $y\in X$ with $d(x,y)<\eps$  and some $n\in\bbn$ such that $d(T^nx, T^ny)>\delta$. The number $\delta$ is called a sensitivity constant for $(X,T)$.

\begin{rem}\label{rem:sensitve-uncountable}
    Note that if $(X,T)$ is sensitive, then $X$ has no isolated points.
    In addition, if $X$ is completely metrizable, then $X$ is uncountable.
\end{rem}

We have the following equivalent conditions of sensitivity.
\begin{lem}\label{lem:sen-equivalent}
    Let $(X, T)$ be a dynamical system with $X$ being completely metrizable.
    Then the following assertions are equivalent:
    \begin{enumerate}
        \item $(X,T)$ is sensitive;
        \item there exists $\delta>0$ such that for every $x\in X$,
              the set
              \[
                  \Bigl\{y\in X\colon \limsup_{n\to\infty}d(T^nx,T^ny)\geq \delta\Bigr\}
              \]
              is a dense $G_\delta$ subset of $X$;
        \item there exists $\delta>0$ such that the set
              \[
                  \Bigl\{(x,y)\in X\times X\colon  \limsup_{n\to\infty}d(T^nx,T^ny)\geq \delta\Bigr\}
              \]
              is a dense $G_\delta$ subset of $X\times X$.
    \end{enumerate}
\end{lem}
\begin{proof}
    (1)$\Rightarrow$(2).
    Let $\delta_0$ be a sensitivity constant for $(X,T)$ and
    $ \delta=\frac{\delta_0}{2}$.
    For any $n\in\bbn$, let
    \[
        X_n=\{y\in X\colon\exists k>n \text{ s.t. } d(T^kx,T^ky)>\delta-\tfrac{1}{n}\}.
    \]
    It is clear that $X_n$ is open.
    Let us show that $X_n$ is dense in $X$.
    Let $U$ be a nonempty open subset of $X$.
    By the continuity of $T$, we can shrink $U$ to a nonempty open subset $V$ such that $\diam(T^i V)<\delta $ for $i=0,1,\dotsc,n$.
    Since $(X,T)$ is sensitive, there exist  $y_1,y_2\in V$ and $k\in \bbn$
    such that $d(T^ky_1,T^ky_2)>\delta_0=2\delta$.
    By the choice of $V$, we have $k>n$.
    By the triangle inequality, one has either $d(T^kx,T^k y_1)>\delta$ or $d(T^k x,T^ky_2)>\delta $. Then either $y_1\in X_n$ or $y_2\in X_n$.
    This shows that $X_n\cap U \neq \emptyset$. Then $X_n$ is dense in $X$, as desired.
    It is easy to check that
    \[
        \Bigl\{y\in X\colon \limsup_{n\to\infty}d(T^nx,T^ny)\geq \delta\Bigr\} = \bigcap_{n=1}^\infty X_n,
    \]
    which is dense in $X$ by the Baire category theorem.

    The implications (2)$\Rightarrow$(3)$\Rightarrow$(1) are obvious.
\end{proof}

Let $(X, T)$ be a dynamical system. 
Following the literature, we say that
a pair $(x,y)\in X\times X$ is \emph{Li-Yorke chaotic}
if
\[
    \liminf_{n\to\infty} d(T^nx,T^ny)=0\qquad \text{and}\qquad \limsup_{n\to\infty} d(T^nx,T^ny)>0,
\]
that is, $(x,y)$ is proximal but not asymptotic.
A subset $K$ of $X$ is called \emph{Li-Yorke scrambled} if any two distinct points $x,y\in K$ form a Li-Yorke chaotic pair.
We say that a dynamical system $(X,T)$ is \emph{Li-Yorke chaotic}
if there exists an uncountable Li-Yorke scrambled subset of $X$.

Usually, the set of all Li-Yorke chaotic pairs is not a $G_\delta$ subset of $X\times X$.
Following Lemma~\ref{lem:dyn-prox-delta}, it is meaningful to consider the set of all Li-Yorke chaotic pairs that are separated by a fixed constant.

For a given positive number $\delta$, a pair $(x,y)\in X\times X$ is called \emph{Li-Yorke $\delta$-chaotic}
if
\[
    \liminf_{n\to\infty} d(T^nx,T^ny)=0\qquad \text{and}\qquad \limsup_{n\to\infty} d(T^nx,T^ny)\geq \delta.
\]
Similarly, we also can define \emph{Li-Yorke $\delta$-scrambled set} and \emph{Li-Yorke $\delta$-chaos}.

Following \cite{AK2003}, we say that a dynamical system $(X,T)$ is \emph{Li-Yorke sensitive} if there exists some $\delta>0$ such that
for any $x\in X$ and $\eps>0$ there exists $y\in X$ with $d(x,y)<\eps$ such that $(x,y)$ is Li-Yorke $\delta$-chaotic.
It is clear that every Li-Yorke sensitive system is sensitive.

Applying the Mycielski theorem (Theorem~\ref{Mycielski}), we have the following criterion for the existence of Li-Yorke $\delta$-scrambled sets in the countable union of open subsets. In addition, if $X$ is separable, then the Li-Yorke $\delta$-scrambled set can be chosen to be dense in $X$.

\begin{prop}\label{prop:dense-delta-scrambled}
    Let $(X, T)$ be a dynamical system  with $X$ being completely metrizable.
    Then the following assertions are equivalent:
    \begin{enumerate}
        \item there exists $\delta>0$ such that for every sequence $(O_j)_j$ of nonempty open subsets of $X$, there exists a sequence $(K_j)_j$ of Cantor sets with $K_j\subset O_j$ such that $\bigcup_{j=1}^\infty K_j$ is Li-Yorke $\delta$-scrambled;
        \item the proximal relation of $(X,T)$ is dense in $X\times X$ and $(X,T)$ is sensitive.
    \end{enumerate}
\end{prop}
\begin{proof}
    (1)$\Rightarrow$(2).
    For any two nonempty open subsets $O_1$ and $O_2$ of $X$, there exist two Cantor sets $K_1\subset O_1$ and $K_2\subset O_2$ such that $K_1\cup K_2$ is Li-Yorke $\delta$-scrambled.
    Then $(K_1\cup K_2)\times (K_1\cup K_2)\subset \prox(T)$ and
    \[
        (K_1\cup K_2)\times (K_1\cup K_2)  \subset  \Bigl\{(x,y)\in X\times X\colon  \limsup_{n\to\infty}d(T^nx,T^ny)\geq \delta\Bigr\}\cup\Delta_X.
    \]
    Since $O_1$ and $O_2$ are arbitrary, $\prox(T)$ is dense in $X\times X$, and by Lemma~\ref{lem:sen-equivalent} $(X,T)$ is sensitive.

    (2)$\Rightarrow$(1).
    Since $(X,T)$ is sensitive, by Lemma~\ref{lem:sen-equivalent} there exists some $\delta>0$ such that
    \[
        R_1= \Bigl\{(x,y)\in X\times X\colon  \limsup_{n\to\infty}d(T^nx,T^ny)
        \geq \delta\Bigr\}
        \]
    is a dense $G_\delta$ subset of $X\times X$.
    Then $R_1\cap \prox(T)$ is also a  dense $G_\delta$ subset of $X\times X$.
    By Remark~\ref{rem:sensitve-uncountable}, $X$ has no isolated points.
    By applying the Mycielski Theorem (Theorem~\ref{Mycielski}) to $R_1\cap \prox(T)$, for every sequence $(O_j)_j$ of nonempty open subsets of $X$, there exists a sequence $(K_j)_j$ of Cantor sets with $K_j\subset O_j$ such that
    \[
        \bigcup_{j=1}^\infty  K_j\times \bigcup_{j=1}^\infty K_j \subset (R_1\cap \prox(T))\cup \Delta_X.
    \]
    It is clear that $\bigcup_{j=1}^\infty K_j$ is Li-Yorke $\delta$-scrambled.
\end{proof}

Motivated by the study of Li-Yorke chaos and distributional chaos for linear operators on Fr\'echet spaces in \cite{BBMP2013, BBMP2015}, we introduce the concept of extreme chaos with respect to a continuous pseudometric.
Indeed, recall that every Fr\'echet space has a separating increasing sequence of seminorms, and that each of these seminorms can be seen as a pseudometric (see Subsections~\ref{subsec:LYeC-group} and~\ref{subsec:LYC-Frechet} for further details regarding this comparison).
From now on in this subsection, we fix a continuous pseudometric $\rho$ on $X$. 
We say that a dynamical system $(X,T)$ is \emph{$\rho$-extremely sensitive}  if for every $x\in X$ and $\eps>0$, there exists $y\in X$ with $d(x,y)<\eps$  and some $n\in\bbn$ such that $\rho(T^nx, T^ny)>\frac{1}{\eps}$.

The proof of the following result is similar to that of Lemma~\ref{lem:sen-equivalent} and we leave it to the reader.

\begin{lem}\label{lem:ext-sen-equivalent}
    Let $(X, T)$ be a dynamical system with $X$ being completely metrizable.
    Then the following assertions are equivalent:
    \begin{enumerate}
        \item $(X,T)$ is $\rho$-extremely sensitive;
        \item for every $x\in X$,
              the set
              \[
                  \Bigl\{y\in X\colon \limsup_{n\to\infty}\rho(T^nx,T^ny)=\infty\Bigr\}
              \]
              is a dense $G_\delta$ subset of $X$;
        \item the set
              \[
                  \Bigl\{(x,y)\in X\times X\colon  \limsup_{n\to\infty}\rho(T^nx,T^ny)=\infty\Bigr\}
              \]
              is a dense $G_\delta$ subset of $X\times X$.
    \end{enumerate}
\end{lem}

Let $(X,T)$ be a dynamical system. 
A pair $(x,y)\in X\times X$ is called \emph{Li-Yorke $\rho$-extremely chaotic}
if
\[
    \liminf_{n\to\infty} d(T^nx,T^ny)=0\qquad \text{and}\qquad \limsup_{n\to\infty} \rho(T^nx,T^ny)=\infty.
\]

Similarly, we can define \emph{Li-Yorke $\rho$-extremely scrambled set} and \emph{Li-Yorke $\rho$-extreme chaos} and \emph{Li-Yorke $\rho$-extreme sensitivity}.

The proof of the following result is analogous to that of Proposition~\ref{prop:dense-delta-scrambled}.

\begin{prop}\label{prop:dense-extreme-scrambled}
    Let $(X,T)$ be a dynamical system  with $X$ being completely metrizable.
    Then the following assertions are equivalent:
    \begin{enumerate}
        \item for every sequence $(O_j)_j$ of nonempty open subsets of $X$, there exists a sequence $(K_j)_j$ of Cantor sets with $K_j\subset O_j$ such that $\bigcup_{j=1}^\infty K_j$ is Li-Yorke $\rho$-extremely scrambled;
        \item the proximal relation of $(X,T)$ is dense in $X\times X$ and $(X,T)$ is $\rho$-extremely sensitive.
    \end{enumerate}
\end{prop}

\begin{rem}
    Suppose that a pseudometric $\rho$ is uniformly continuous, that is, for every $\eps>0$ there exists $\delta>0$ such that for any $x,y\in X$ with $d(x,y)<\delta$, one has $\rho(x,y)<\eps$. So in this case, extreme sensitivity implies sensitivity, and Li-Yorke extreme chaos implies Li-Yorke $\delta$-chaos for some $\delta>0$.
\end{rem}

\subsection{Li-Yorke chaos for continuous endomorphisms of completely metrizable groups}
\label{subsec:LYC-group}
In \cite{BBMP2011} and \cite{BBMP2015} the authors characterized Li-Yorke chaos for linear operators on Banach and Fr\'echet spaces.
As explained in the introduction, one of the motivations of this paper is using the techniques from topological dynamics to show that many results also hold in the general setting of continuous endomorphisms of completely metrizable groups.

In this subsection, unless otherwise specified, $G$ denotes a completely metrizable group endowed with a left-invariant compatible metric $d$.
Let $T\in\eend(G)$, that is, $T$ is a continuous endomorphism of the completely metrizable group $G$.
Then $(G, T)$ forms a dynamical system and, when stating dynamical properties of $(G, T)$, we will simplify the notation only referring to the endomorphism $T$.

In \cite{B1988}, Beauzamy introduced the concept of irregular vectors for continuous linear operators on Banach spaces.
In \cite{BBMP2011}, Berm\'{u}dez et al.\ characterized Li-Yorke chaos in terms of the existence of irregular vectors for continuous linear operators on Banach spaces. Later in \cite{BBMP2015}, Bernardes et al.\ introduced the concept of semi-irregular vectors for continuous linear operators on Fr\'echet spaces.
In this paper, we introduce semi-irregular points for continuous endomorphisms of groups.
Let $T\in\eend(G)$. 
A point $x\in G$ is called \emph{semi-irregular} for $T$ if
\[
    \liminf_{n\to\infty} d(T^nx,e)=0\qquad \text{and}\qquad \limsup_{n\to\infty} d(T^nx,e)>0,
\]
that is, $(x,e)$ is a Li-Yorke chaotic pair.

Based on the left-invariance of the metric, we have the following auxiliary result.
\begin{lem}\label{lem:g-asym-prox-left-inv}
    Let $T\in\eend(G)$.
    \begin{enumerate}
        \item For every $x\in G$,
              $\asym(T,x)=x\cdot \asym(T,e)$,
              $\prox(T,x)=x\cdot \prox(T,e)$.
        \item $\asym(T,e)$ is dense in $G$ if and only if
              $\asym(T)$ is dense in $G\times G$.
        \item If $\asym(T,e)$ is residual in $G$, then $\asym(T,e)=G$.
        \item $\prox(T,e)$ is dense in $G$ if and only if $\prox(T)$ is dense in $ G\times G$.
        \item For every $x,y\in G$, $(x,y)$ is a Li-Yorke chaotic pair if and only if $x^{-1}y$ is a semi-irregular point.
        \item If $K$ is a  Li-Yorke scrambled subset of $G$,
              then for any $x\in G$, $xK$ is also a Li-Yorke scrambled set. In particular, for any $x\in K$, $x^{-1}K\setminus \{e\}$ consists of semi-irregular points.
        \item For a given $\delta>0$, if $K$ is a Li-Yorke $\delta$-scrambled subset of $G$,
              then for any  $x\in G$, $xK$ is also a Li-Yorke $\delta$-scrambled set.
    \end{enumerate}
\end{lem}
\begin{proof}
    Since $d$ is left-invariant, for every $x,y\in G$ and $n\in\bbn$, one has
    \[
        d(T^nx,T^ny)= d(T^n(x^{-1})T^nx,T^n(x^{-1})T^n y)
        =d(e,T^n(x^{-1}y)).
    \]
    It is easy to see that (1), (5), (6) and (7) follow from the above formula.

    (2). If $\asym(T,e)$ is dense in $G$, then by (1) for every $x\in G$,
    $\asym(T,x)$ is also dense in $G$. As $\asym(T)=\bigcup_{x\in G}\{x\}\times \asym(T,x)$, $\asym(T)$ is dense in $G\times G$.

    Now assume that $\asym(T)$ is dense in $G\times G$. For every nonempty open subset $W$ of $G$, as $G$ is a topological group, there exists an open neighborhood $U$ of $e$ and a nonempty open subset $V$ of $W$ such that $U^{-1}\cdot V\subset W$. Pick a point $(x,y)\in \asym(T)\cap (U\times V)$.
    Then $(e,x^{-1}y)\in \asym(T)$. Therefore, $x^{-1}y\in \asym(T,e)\cap W$, which implies that $\asym(T,e)$ is dense in $G$.

    (3). For every $x\in G$, $\asym(T,e)x^{-1}$ is also residual in $G$.
    Take $y\in \asym(T,e)\cap \asym(T,e)x^{-1}$. Then $(y,e)$ and $(yx,e)$ are asymptotic. Since the asymptotic relation is transitive, we have $(y,yx)\in \asym(T)$ and then $(e,x)\in \asym(T)$, Hence, $x\in \asym(T,e)$.

    (4). The proof is similar to (2).
\end{proof}

\begin{lem}\label{lem:subseq-asym-subgroup}
    Let $T\in\eend(G)$ and
    $(n_k)_k$ be an increasing sequence of positive integers.
    Then
    \[
        G_0=\Bigl\{x\in G\colon \lim_{k\to \infty}T^{n_k}x=e\Bigr\}
    \] is a $T$-invariant subgroup of $G$.
\end{lem}
\begin{proof}
    Since $T$ is continuous and $Te=e$, one can easily verify that
    $G_0$ is  $T$-invariant.
    For every $x,y\in G_0$, we have

    \begin{equation*}
        \begin{split}
            \limsup_{k\to \infty}d(T^{n_k}(y^{-1}x),e)&=
            \limsup_{k\to \infty}d(T^{n_k}x,T^{n_k}y)\\
            &\leq\lim_{k\to \infty}d(T^{n_k}x,e)+\lim_{k\to\infty} d(e,T^{n_k}y)=0.
        \end{split}
    \end{equation*}
    Then $G_0$ is a subgroup of $G$.
\end{proof}

Auslander and Yorke proved in~\cite{AY1980} that if a compact dynamical system is minimal (that is, it does not contain any proper closed invariant subsets), then it is either equicontinuous or sensitive.
We also have the following dichotomy of equicontinuity and sensitivity for continuous endomorphisms of completely metrizable groups.
\begin{thm}\label{thm:dich-eq-sen}
    Let $T\in\eend(G)$.
    Then $T$ is either equicontinuous or sensitive.
\end{thm}
\begin{proof}
    Let $e$ be the identity of $G$. We have the following two cases:

    Case 1: for any $\eps>0$, there exists $\delta>0$ such that for every $x\in G$ with $d(x,e)<\delta$,
    we have $d(T^nx,e)<\eps$ for all $n\in\bbn$.
    Let $x_1, x_2\in G$.
    If $d(x_1,x_2)<\delta$, by the left-invariance of $d$ we have $d(x_1,x_2)=d(x_2^{-1}x_1,e)<\delta$. Then for every $n\in\bbn$, we have 
    \begin{align*}
        d(T^nx_1,T^nx_2)=d( T^n (x_2^{-1}) T^n x_1, e)
        = d( T^n (x_2^{-1} x_1), e)<\eps.
    \end{align*}
    This implies that $T$ is equicontinuous.

    Case 2: there exists $\delta>0$ such that for every $\eps>0$ one can find $y_\eps\in G$  with $d(y_\eps,e)<\eps$ and $n\in\bbn$ such that $d(T^ny_\eps,e)>\delta$.
    For every $x\in G$ and $\eps>0$,  we have
    $d(xy_\eps,x)<\eps$ and
    \[
        d(T^n(xy_\eps), T^n x)=d(T^n(x^{-1}) T^n(xy_\eps), e)
        = d(T^n y_\eps,e)>\delta.
    \]
    Thus $T$ is sensitive.
\end{proof}

Combining Lemma~\ref{lem:eq-prox-asym} and Theorem~\ref{thm:dich-eq-sen} we have the following.
\begin{coro}\label{cor:scrambled-to-sensitive}
    Let $T\in\eend(G)$.
    If there exists a Li-Yorke chaotic pair, then $T$ is sensitive.
\end{coro}

Now we give a characterization of the existence of a  dense set of semi-irregular points for continuous endomorphisms of completely metrizable groups.

\begin{thm}\label{thm:equi-of-dense-LY}
    Let $T\in\eend(G)$.
    Then the following assertions are equivalent:
    \begin{enumerate}
        \item $T$ admits a dense set of semi-irregular points;
        \item $T$ admits a residual set of semi-irregular points;
        \item there exists $\delta>0$ such that for every sequence $(O_j)_j$ of nonempty open subsets of $G$, there exists a sequence $(K_j)_j$ of Cantor sets with $K_j\subset O_j$ such that $\bigcup_{j=1}^\infty K_j$ is Li-Yorke $\delta$-scrambled;
        \item the proximal relation of $T$ is dense in $G\times G$ and $T$ is sensitive;
        \item the proximal cell of $e$ is dense in $G$ and there exists $x\in G$ such that
              \[
                  \limsup_{n\to\infty} d(T^nx,e)>0.
              \]
    \end{enumerate}
\end{thm}
\begin{proof}

    The implications (2) $\Rightarrow$(1)$\Rightarrow$(5) are clear.

    (5)$\Rightarrow$(4). By Lemma~\ref{lem:dyn-prox-delta}~(1), $\prox(T,e)$ is residual in $G$.  If $\prox(T,e)=\asym(T,e)$, then by Lemma~\ref{lem:g-asym-prox-left-inv}~(3), one has $\asym(T,e)=G$, which is a contradiction. So $\prox(T,e)\setminus \asym(T,e)\neq\emptyset$, that is, $\prox(T,e)$ contains a semi-irregular point. Now by Corollary~\ref{cor:scrambled-to-sensitive}, $T$ is sensitive.

    (3) $\Leftrightarrow$ (4). It follows from Proposition~\ref{prop:dense-delta-scrambled}.

    (4) $\Rightarrow$ (2).
    As $T$ is sensitive, by Lemma~\ref{lem:sen-equivalent} there exists $\delta>0$ such that
    \[
        \Bigl\{x \in G\colon \limsup_{n\to\infty} d(T^nx,e)\geq \delta\Bigr\}
    \]
    is a dense $G_\delta$ subset of $G$.
    Let
    \[
        G_0=\prox(T,e)\cap \Bigl\{x \in G\colon \limsup_{n\to\infty} d(T^nx,e)\geq \delta\Bigr\}.
    \]
    Then $G_0$ is a dense $G_\delta$ subset of $G$ consisting of semi-irregular points.
\end{proof}

According to Theorem~\ref{thm:equi-of-dense-LY}\,(5), we have the following consequence.
\begin{coro}
    Let $T\in\eend(G)$. If the proximal cell of $e$ is dense in $G$ and $T$ has a non-trivial periodic point, then it has a dense set of semi-irregular points.
\end{coro}

By applying Theorem~\ref{thm:equi-of-dense-LY},
we also can characterize Li-Yorke chaos
for continuous endomorphisms of completely metrizable groups.
\begin{thm}\label{thm:equi-of-LYC}
    Let $T\in\eend(G)$.
    Then the following assertions are equivalent:
    \begin{enumerate}
        \item $T$ admits a semi-irregular point;
        \item $T$ admits a Li-Yorke chaotic pair;
        \item $T$ is Li-Yorke chaotic;
        \item $T$ is Li-Yorke sensitive;
        \item the restriction of $T$ to some invariant closed subgroup $\widetilde {G}$ has a dense set of semi-irregular points.
    \end{enumerate}
\end{thm}
\begin{proof}
    The implications (3)$\Rightarrow$(2)$\Rightarrow$(1) and (4)$\Rightarrow$(2) are clear.

    (5)$\Rightarrow$(3). It follows from Theorem~\ref{thm:equi-of-dense-LY}.

    (1)$\Rightarrow$(5). Let $x$ be a semi-irregular point.
    Then there exists an increasing sequence $(n_k)_k$ of positive integers 
    such that $ \lim_{k\to \infty}T^{n_k}x=e$.
    Let
    \[
        G_0=  \Bigl\{y\in G\colon  \lim_{k\to \infty}T^{n_k}y= e\Bigr\}.
    \]
    It is clear that $x\in G_0$.
    By Lemma~\ref{lem:subseq-asym-subgroup}, $G_0$ is a $T$-invariant subgroup of $G$.
    Let $\widetilde{G}$ be the closure of $G_0$.
    Then $\widetilde{G}$ is also a $T$-invariant subgroup of $G$.
    As $(x,e)$ is Li-Yorke chaotic for $T|_{\widetilde{G}}$, by Corollary~\ref{cor:scrambled-to-sensitive} $T|_{\widetilde{G}}$ is sensitive.
    It is clear that $G_0\subset \prox(T|_{\widetilde{G}},e)$.
    So $\prox (T|_{\widetilde{G}})$ is dense in $\widetilde{G}\times \widetilde{G}$. Now (5) follows from Theorem~\ref{thm:equi-of-dense-LY}.

    (5)$\Rightarrow$(4).
    By applying Theorem~\ref{thm:equi-of-dense-LY} to $T|_{\widetilde{G}}$, there exists $\delta>0$ such that \[
        G_0=\prox(T,e)\cap \Bigl\{x \in \widetilde{G}\colon \limsup_{n\to\infty} d(T^nx,e)\geq \delta\Bigr\}
    \]
    is residual in $\widetilde{G}$.
    For every $x\in G$ and $y\in xG_0$, $(x,y)$ is Li-Yorke $\delta$-chaotic. Note that $x$ is an accumulation point of $xG_0$. Then $T$ is Li-Yorke sensitive.
\end{proof}

\begin{rem}\label{rem:separable-G-LY-chaos}
    (a) If $G$ is separable in Theorem~\ref{thm:equi-of-dense-LY}, then the assertion (3) can be replaced by
    the following one:
    \begin{enumerate}
        \item[(3$'$)] there exists a dense $\sigma$-Cantor Li-Yorke $\delta$-scrambled subset of $G$ for some $\delta>0$.
    \end{enumerate}

    (b) In the proof of  (1)$\Rightarrow$(5) in Theorem~\ref{thm:equi-of-LYC}, if we let $G_1$ be the countable group generated by the orbit of $x$, then $G_1$ is a $T$-invariant subgroup of $G_0$.
    Let $\widetilde{G_1}$ be the closure of $G_1$.
    Then $\widetilde{G_1}$ is a $T$-invariant closed separable subgroup of $G$.  So the assertion (5) in Theorem~\ref{thm:equi-of-LYC}  can be replaced by the following one:
    \begin{enumerate}
        \item[(5$'$)] the restriction of $T$ to some invariant closed separable subgroup $\widetilde {G}$ has a dense set of semi-irregular points.
    \end{enumerate}
\end{rem}

Due to the group structure, the following result shows that the existence of  a residual Li-Yorke scrambled set implies that the whole space is a Li-Yorke scrambled set.

\begin{prop}\label{prop:residul-scrambled-set}
    Let $T\in\eend(G)$.
    Then the following assertions are equivalent:
    \begin{enumerate}
        \item $T$ admits a residual Li-Yorke scrambled set;
        \item every point in $G\setminus \{e\}$ is semi-irregular for $T$;
        \item $X$ is a Li-Yorke scrambled set for $T$.
    \end{enumerate}
\end{prop}
\begin{proof}
    The implications (2) $\Rightarrow$ (3) $\Rightarrow$(1) are clear.

    (1) $\Rightarrow$ (2).
    Let $K$ be a residual Li-Yorke scrambled subset of $G$.
    For every $x\in G\setminus \{e\}$, $K\cap Kx^{-1}$ is also residual.
    Take $y\in K\cap Kx^{-1}$. One has that $(y,yx)$ is Li-Yorke chaotic. Then $(e,x)$ is also Li-Yorke chaotic, which shows that $x$ is a semi-irregular point.
\end{proof}

\begin{rem}
It is worth noticing that a point $x\in G$ is semi-irregular for a continuous endomorphism $T$ if and only if the orbit of $x$ has a subsequence that is convergent to the identity $e$ but the whole orbit is not convergent. 
So, being semi-irregular is a topological property and is not dependent on the compatible metric.
Even though we fix a left-invariant compatible metric $d$ on $G$ at the beginning of this subsection, the Li-Yorke chaos in Theorems~\ref{thm:equi-of-dense-LY} and~\ref{thm:equi-of-LYC} holds true for any left-invariant compatible metric on $G$.
\end{rem}

Now we present two examples of inner automorphisms of  Polish groups which are densely Li-Yorke chaotic.

\begin{exam}\label{exam:SZ-LY-chaos}
Let $\bbs(\bbz)$ be the group of all permutations of the integers $\bbz$. One can see that $\bbs(\bbz)$ is a Polish topological group 
with respect to the topology induced by pointwise convergence. 
Let $\sigma\in \bbs(\bbz)$ be the unit shift permutation, that is, $\sigma(n)=n-1$ for each $n\in\bbz$.
Define $T\colon \bbs(\bbz)\to \bbs(\bbz)$ by $T(f)=\sigma f \sigma^{-1}$.
Then for every $f\in \bbs(\bbz)$, $n,k\in\bbz$, we have $T^n(f)(k)=f(k+n)-n$. 
It is shown in \cite[Example 1]{M2009} that $(\bbs(\bbz),T)$ is mixing.

Let $A$ be the subgroup of $\bbs(\bbz)$ consisting of the permutations fixing all but a finite set in $\bbz$. 
It is easy to see that $A$ is dense in $\bbs(\bbz)$ and for any $f\in A$, $T^n(f)$ converges in $\bbs(\bbz)$ to the identity permutation as $n\to\infty$.
Then $A\times A$ is contained in the asymptotic relation with respect to any
 compatible metric on $\bbs(\bbz)$.
It is clear that $T(\sigma)=\sigma$. 
Then, by Theorem~\ref{thm:equi-of-dense-LY} and Remark \ref{rem:separable-G-LY-chaos}, $(\bbs(\bbz), T)$ is densely Li-Yorke chaotic for any left-invariant compatible metric on $\bbs(\bbz)$.

A precise left-invariant compatible metric $d$ on $\bbs(\bbz)$ is defined by 
$d(f,g)=\frac{1}{2^k}$ where $k=\min \{|n| \colon f(n)\neq g(n), n\in\bbz\}$ for any $f\neq g\in \bbs(\bbz)$.
Define $\rho\colon\bbs(\bbz)\times \bbs(\bbz) \to\bbr$ by  $\rho(f,g)=|f(0)-g(0)|$.
Then, $\rho$ is a continuous pseudometric on $\bbs(\bbz)$.
One can easily see that  $(\bbs(\bbz), T)$ is $\rho$-extremely sensitive.
By Proposition~\ref{prop:dense-extreme-scrambled}, $(\bbs(\bbz), T)$ is densely Li-Yorke $\rho$-extremely chaotic.
\end{exam}

\begin{exam}\label{exam:HR-LY-chaos}
Let $H_+(\bbr)$ be the group of all increasing homeomorphisms of $\bbr$. 
$H_+(\bbr)$ is a Polish topological group with respect to the compact-open topology. 
Let $\sigma\in H_+(\bbr)$ be the unit translation, that is, $\sigma(x)=x-1$ for each $x\in\bbr$.
Define $T\colon H_+(\bbr)\to H_+(\bbr)$ by $T(f)=\sigma f \sigma^{-1}$.
Then, for every $f\in H_+(\bbr)$, $n\in\bbz$ and $x\in\bbr$, we have $T^n(f)(x)=f(x+n)-n$. 
It is shown in \cite[Example 3]{M2009} that $(H_+(\bbr),T)$ is mixing.

Let $A$ be the subgroup of $H_+(\bbr)$  consisting of the increasing homeomorphisms fixing all but a bounded set in $\bbr$. 
It is easy to see that $A$ is dense in $H_+(\bbr)$ and for any $f\in A$, $T^n(f)$ converges in $H_+(\bbr)$ to the identity  as $n\to\infty$.
Then $A\times A$ is contained in the asymptotic relation with respect to any compatible metric on $H_+(\bbr)$.
It is clear that $T(\sigma)=\sigma$. 
Then, by Theorem~\ref{thm:equi-of-dense-LY} and Remark \ref{rem:separable-G-LY-chaos}, $(H_+(\bbr), T)$ is densely Li-Yorke chaotic for any left-invariant compatible metric on $H_+(\bbr)$.

Precisely, a left-invariant compatible metric $d$ on $H_+(\bbr)$ is defined by 
$d(f,g)=2^{-a}$, where $a=\inf \{|x|\colon f(x)\neq g(x), x\in\bbr\}$ for any $f\neq g\in H_+(\bbr)$.
Define $\rho\colon H_+(\bbr)\times H_+(\bbr) \to\bbr$ by  $\rho(f,g)=|f(0)-g(0)|$.
Then $\rho$ is a continuous pseudometric on $H_+(\bbr)$.
It is easy to see that  $(H_+(\bbr), T)$ is $\rho$-extremely sensitive.
By Proposition~\ref{prop:dense-extreme-scrambled}, $(H_+(\bbr), T)$ is densely Li-Yorke $\rho$-extremely chaotic.
\end{exam}

\subsection{Li-Yorke extreme chaos for continuous endomorphisms of completely metrizable groups}
\label{subsec:LYeC-group}
With the same notation as in the previous subsection,
$G$ denotes a completely
metrizable group endowed with a left-invariant compatible metric $d$.
Since every Fr\'echet space has a separating increasing sequence of seminorms, for metrizable groups we consider its counterpart, the pseudometric.
In this subsection we also fix a left-invariant continuous pseudometric $\rho$ on $G$.

Let $T\in\eend(G)$. 
A point $x\in G$ is called \emph{irregular} for $T$ if
\[
    \liminf_{n\to\infty} d(T^nx,e)=0\qquad \text{and}\qquad \limsup_{n\to\infty} \rho(T^nx,e)=\infty,
\]
that is, $(x,e)$ is a Li-Yorke extremely chaotic pair.

Similar to Lemma~\ref{lem:g-asym-prox-left-inv}, we have the following  observation.

\begin{lem} \label{lem:g-ext-left-inv}
    Let $T\in\eend(G)$.
    \begin{enumerate}
        \item For every $x,y\in G$, $(x,y)$ is a Li-Yorke extremely chaotic pair if and only if $x^{-1}y$ is an irregular point.
        \item If $K$ is a Li-Yorke extremely scrambled subset of $G$,
              then for any $x\in G$, $xK$ is also a Li-Yorke extremely scrambled set. In particular, for any $x\in K$, $x^{-1}K\setminus \{e\}$ consists of  irregular points.
    \end{enumerate}
\end{lem}

\begin{prop}\label{prop:g-ext-sen-cond}
    Let $T\in\eend(G)$. Then $T$ is extremely sensitive if and only if there exists a sequence $(y_k)_k$ in $G$ such that $\lim_{k\to\infty}y_k=e$ and $\limsup_{k\to\infty}\rho(T^{k}y_k,e)=\infty$.
\end{prop}
\begin{proof}
    The "only if" implication follows immediately from the definition. Now we prove the "if" implication. For any $\eps>0$ and $x\in G$, take $k\in\bbn$ such that $ d(y_k,e)<\eps$
    and $\rho(T^k y_k,e)>\frac{1}{\eps}$.
    By the left-invariance of $d$, we have $d(xy_k,x)=d(y_k,e)<\eps$ and $\rho(T^k(xy_k),T^kx)=\rho(T^k y_k,e)>\frac{1}{\eps}$.
    Then $T$ is extremely sensitive.
\end{proof}

\begin{prop} \label{prop:irregular-ext-sen}
    Let $T\in\eend(G)$.
    If there exists an irregular point, then $T$ is extremely sensitive.
\end{prop}
\begin{proof}
    Let $y\in G$ be an irregular point. Then there exists a sequence $(i_k)_k$ in $\bbn$ such that $\lim_{k\to\infty}T^{i_k}y=e$ and
    $\limsup_{n\to\infty} \rho(T^ny,e)=\infty$.
    Note that for any $k\in\bbn$,
    $\limsup_{n\to\infty} \rho(T^n(T^{i_k}y),e)=\infty$.
    Then, by Proposition~\ref{prop:g-ext-sen-cond}, $T$ is extremely sensitive.
\end{proof}

Similar to Theorem~\ref{thm:equi-of-dense-LY}, we have the following  characterization for the existence of a dense set of irregular points for continuous endomorphisms of completely metrizable groups.
\begin{thm} \label{thm:equi-of-dense-irr-points}
    Let $T\in\eend(G)$.
    Then the following assertions are equivalent:
    \begin{enumerate}
        \item $T$ admits a dense set of irregular points;
        \item $T$ admits a residual set of irregular points;
        \item for every sequence $(O_j)_j$ of nonempty open subsets of $G$, there exists a sequence $(K_j)_j$ of Cantor sets with $K_j\subset O_j$ such that $\bigcup_{j=1}^\infty K_j$ is Li-Yorke extremely scrambled;
        \item the proximal relation of $T$ is dense in $G\times G$ and $T$ is extremely sensitive.
    \end{enumerate}
\end{thm}
\begin{proof}
    The implication (2)$\Rightarrow$(1) is clear.

    (1)$\Rightarrow$(4). Let $A$ be a dense set of irregular points.
    Then $A\subset \prox(T,e)$.
    By Lemma~\ref{lem:g-asym-prox-left-inv}, $\prox(T)$ is dense in  $G\times G$. Lemma~\ref{prop:irregular-ext-sen} ensures that $T$ is extremely sensitive.

    (3)$\Leftrightarrow$(4). It follows from Proposition~\ref{prop:dense-extreme-scrambled}.

    (4)$\Rightarrow$(2). The set of irregular points  equals to the intersection
    \[
        \prox(T,e)\cap  \Bigl\{x\in G\colon \limsup_{n\to\infty}\rho(T^nx,e)=\infty\Bigr\}.
    \]
    By Lemma~\ref{lem:g-asym-prox-left-inv}, $\prox(T,e)$ is a dense $G_\delta$ subset of $G$. Since $T$ is extremely sensitive, by Lemma~\ref{lem:ext-sen-equivalent}, the set
    \[
        \Bigl\{x\in G\colon \limsup_{n\to\infty}\rho(T^nx,e)=\infty\Bigr\}
    \]
    is also a dense $G_\delta$ subset of $G$.
    Then the set of all irregular points is residual.
\end{proof}

Using the same arguments as in the proof of Theorem~\ref{thm:equi-of-LYC}, we can apply Theorem~\ref{thm:equi-of-dense-irr-points} to obtain a characterization of Li-Yorke extreme chaos for continuous endomorphisms of completely metrizable groups.
\begin{thm}\label{thm:equi-of-irr-point}
    Let $T\in\eend(G)$.
    Then the following assertions are equivalent:
    \begin{enumerate}
        \item $T$ admits an irregular point;
        \item $T$ admits a Li-Yorke extremely chaotic pair;
        \item $T$ is Li-Yorke extremely chaotic;
        \item $T$ is Li-Yorke extremely sensitive;
        \item the restriction of $T$ to some invariant closed subgroup $\widetilde {G}$ has a dense set of irregular points.
    \end{enumerate}
\end{thm}

\begin{rem}
    (a) If $G$ is separable in Theorem~\ref{thm:equi-of-dense-irr-points}, then the assertion (3) can be replaced by
    the following one:
    \begin{enumerate}
        \item[(3$'$)] there exists a dense $\sigma$-Cantor Li-Yorke  extremely scrambled subset of $G$.
    \end{enumerate}

    (b) In the proof of  (1)$\Rightarrow$(5) in Theorem~\ref{thm:equi-of-irr-point}, if we let $G_1$ be the countable group generated by the orbit of an irregular point $x$, then $G_1$ is a $T$-invariant subgroup of $G_0$.
    Let $\widetilde{G_1}$ be the closure of $G_1$.
    Then $\widetilde{G_1}$ is a $T$-invariant closed separable subgroup of $G$.  So the assertion (5) in  Theorem~\ref{thm:equi-of-irr-point} can be replaced by the following one:
    \begin{enumerate}
        \item[(5$'$)] the restriction of $T$ to some invariant closed separable subgroup $\widetilde {G}$ has a dense set of irregular points.
    \end{enumerate}
\end{rem}

\subsection{Li-Yorke chaos for continuous linear operators on Fr{\'e}chet spaces} \label{subsec:LYC-Frechet}
In this subsection, unless otherwise specified, $X$ denotes an arbitrary infinite-dimensional Fr{\'e}chet space.
Let $(p_j)_j$ be a separating increasing sequence of seminorms endowing the topology of $X$. The Fr{\'e}chet-space metric $d$ on $X$ is given by
\[
    d(x,y):=\sum_{j=1}^\infty \frac{1}{2^j} \min(1,p_j(x-y)),\quad x,y\in X.
\]
Even though Li-Yorke chaos for linear operators on Banach and Fr\'echet spaces has been characterized in \cite{BBMP2011} and \cite{BBMP2015},
here we will use the results from Subsections~\ref{subsec:LYC-group} and~\ref{subsec:LYeC-group} to reobtain some of the results from
\cite{BBMP2011} and \cite{BBMP2015}, and also to slightly generalize some of them. 

First, we have the following auxiliary result.
\begin{lem}\label{lem:subseq-asym-subspace}
    Let $T\in L(X)$ and
    $(n_k)_{k}$ be an increasing sequence of positive integers.
    Then
    \[
        X_0:=\Bigl\{x\in X\colon \lim_{k\to \infty}T^{n_k}x=\vecz\Bigr\}
    \] is a $T$-invariant subspace  of $X$.
\end{lem}
\begin{proof}
    Since $T$ is continuous and $T\vecz=\vecz$, it is easy verify that
    $X_0$ is $T$-invariant.
    For every $x,y\in X_0$ and $\alpha,\beta \in \bbk$, we have
    \begin{equation*}
        \begin{split}
            \lim_{k\to \infty}T^{n_k}(\alpha x+\beta y)&=\alpha\lim_{k\to \infty}T^{n_k} x+\beta \lim_{k\to \infty}T^{n_k}y =\vecz.
        \end{split}
    \end{equation*}
    Thus $X_0$ is a subspace of $X$.
\end{proof}

Recall that $T\in L(X)$ is equicontinuous as a dynamical property if and only if the family $\{T^n\colon n\in\bbn\}$ of operators is equicontinuous in the theory of functional analysis.
The next result follows directly from the well-known Banach-Steinhaus Theorem.

\begin{prop}\label{prop:f-space-eq}
    Let $T\in L(X)$. Then $T$ is equicontinuous if and only if for every $x\in X$, the orbit of $x$ is bounded.
\end{prop}

\begin{rem}
    Assume in addition that $X$ is a Banach space.
    Then $T\in L(X)$ is equicontinuous if and only if $T$ is power-bounded, that is, $\sup_n \Vert T^n\Vert<\infty$.
\end{rem}

Each seminorm $p_j$ induces a continuous left-invariant pseudometric $d_j$ defined by $d_j(x,y)=p_j(x-y)$.
Let $T\in L(X)$ and $m\in\bbn$. A vector $x\in X$ is called \emph{$m$-irregular} if it is irregular with respect to the seminorm $p_m$.
A pair $(x,y)\in X\times X$ is called \emph{Li-Yorke $m$-extremely chaotic} if it is Li-Yorke extremely chaotic with respect to the seminorm $p_m$.
Similarly, we can define \emph{$m$-extreme sensitivity}, \emph{Li-Yorke $m$-extremely scrambled set}, \emph{Li-Yorke $m$-extreme chaos} and
\emph{Li-Yorke $m$-extreme sensitivity}.

\begin{rem}
    Let $(X,\lVert\cdot \rVert)$ be a Banach space and $T\in L(X)$.
    In this case, the sequence of seminorms $(p_j)_j$ is just the norm $\lVert\cdot \rVert$ on $X$.  Then
    $x\in X$ is an irregular vector if and only if
    \[
        \liminf_{n\to\infty} \lVert T^nx\rVert =0\qquad \text{and}\qquad \limsup_{n\to\infty} \lVert T^nx\rVert=\infty.
    \]
    Moreover, a pair $(x,y)\in X\times X$ is Li-Yorke extremely chaotic if and only if
    \[
        \liminf_{n\to\infty} \lVert T^nx-T^ny\rVert=0\qquad \text{and}\qquad \limsup_{n\to\infty} \lVert T^nx-T^ny\rVert=\infty.
    \]
\end{rem}

\begin{prop}\label{prop:f-space-sen}
    Let $T\in L(X)$.
    Then the following assertions are equivalent:
    \begin{enumerate}
        \item $T$ is sensitive;
        \item there exists an unbounded orbit;
        \item there exists a bounded sequence $(y_k)_k$ in $X$ such that
              $(T^ky_k)_k$ is unbounded;
        \item $T$ is $m$-extremely sensitive for some $m\in\bbn$.
    \end{enumerate}
\end{prop}
\begin{proof}
    (1)$\Leftrightarrow$(2). It follows immediately from Theorem~\ref{thm:dich-eq-sen} and Proposition~\ref{prop:f-space-eq}.

    The implications (2)$\Rightarrow$(3)  and $(4)\Rightarrow$(1) are clear.

    (3)$\Rightarrow$(4). Since $(T^ky_k)_k$ is unbounded, there exists $m\in\bbn$ such that $\sup_k p_m(T^ky_k)=\infty$.
    For each $n>0$, there exists $N\in\bbn$ such that $d(\frac{1}{N} y_k,\vecz)<\frac{1}{n}$ for all $k\in\bbn$. Take $k\in\bbn$ such that $p_m(T^ky_k)>nN$.
    Then $p_m(T^k(\frac{1}{N}y_k))>n$. By Proposition~\ref{prop:g-ext-sen-cond}, $T$ is $m$-extremely sensitive.
\end{proof}

Combining Proposition~\ref{prop:f-space-sen} and Theorems~\ref{thm:equi-of-dense-LY} and \ref{thm:equi-of-dense-irr-points}, we have the following characterization of the existence of a dense set of semi-irregular vectors for continuous linear operators on Fr\'echet spaces, which generalizes 
\cite[Theorem 10]{BBMP2015} from separable Fr\'echet spaces to general Fr\'echet spaces.

\begin{thm}\label{thm:F-space-dense-LY-chaos}
    Let $T\in L(X)$. Then the following assertions are equivalent:
    \begin{enumerate}
        \item $T$ admits a dense set of semi-irregular vectors;
        \item $T$ admits a residual set of $m$-irregular vectors for some $m\in\bbn$;
        \item there exists $m\in\bbn$ such that for every sequence $(O_j)_j$ of nonempty open subsets of $X$, there exists a sequence $(K_j)_j$ of Cantor sets with $K_j\subset O_j$ such that $\bigcup_{j=1}^\infty K_j$ is Li-Yorke $m$-extremely scrambled;
        \item the proximal relation of $T$ is dense in $X\times X$ and
              $T$ is sensitive;
        \item the proximal cell of $\vecz$ is dense in $X$ and there exists $x\in X$ such that
              \[\limsup_{n\to\infty}d(T^nx,\vecz)>0.\]
    \end{enumerate}
\end{thm}

By applying Theorem~\ref{thm:F-space-dense-LY-chaos}, we obtain the following characterization of Li-Yorke $m$-extreme chaos for continuous linear operators on Fr\'echet spaces, which slightly strengthens 
\cite[Theorem 9]{BBMP2015}. Moreover, as far as we know, the assertion (5) is new.

\begin{thm}\label{thm:f-space-LY-chaos}
    Let $T\in L(X)$. Then the following assertions are equivalent:
    \begin{enumerate}
        \item $T$ admits a semi-irregular vector;
        \item $T$ admits a Li-Yorke chaotic pair;
        \item $T$ admits an $m$-irregular vector for some $m\in\bbn$;
        \item $T$ is Li-Yorke $m$-extremely chaotic for some $m\in\bbn$;
        \item $T$ is Li-Yorke $m$-extremely sensitive for some $m\in\bbn$;
        \item the restriction of $T$ to some closed $T$-invariant subspace $\widetilde{X}$ has a residual set of $m$-irregular points for some $m\in\bbn$.
    \end{enumerate}
\end{thm}
\begin{proof}
    The implications (4) $\Rightarrow$(3)$\Rightarrow$(2)$\Rightarrow$(1)
    and (5) $\Rightarrow$(1) are clear.

    (6)$\Rightarrow$(4). It follows from Theorem~\ref{thm:F-space-dense-LY-chaos}.

    (1)$\Rightarrow$(6).
    Let $x$ be a semi-irregular vector.
    There exists an increasing sequence of positive integers $(n_k)$
    such that $ \lim_{k\to \infty}T^{n_k}x=\vecz$.
    Let
    \[ X_0=  \Bigl\{y\in X\colon  \lim_{k\to \infty}T^{n_k}y=\vecz\Bigr\}.\]
    It is clear that $x\in X_0$.
    By Lemma~\ref{lem:subseq-asym-subspace} $X_0$ is a $T$-invariant subspace of $X$.
    Let $\widetilde{X}$ be the closure of $X_0$.
    Then $\widetilde{X}$ is also a $T$-invariant subspace of $X$.
    As $(x,e)$ is Li-Yorke chaotic for $T|_{\widetilde{X}}$, by Corollary~\ref{cor:scrambled-to-sensitive} $T|_{\widetilde{X}}$ is sensitive.
    It is clear that $X_0\subset \prox(T|_{\widetilde{X}},e)$.
    So $\prox (T|_{\widetilde{X}})$ is dense in $\widetilde{X}\times \widetilde{X}$. Now (6) readily follows from  Theorem~\ref{thm:F-space-dense-LY-chaos}.

    (6)$\Rightarrow$(5).
    By applying Theorem~\ref{thm:F-space-dense-LY-chaos} to $T|_{\widetilde{X}}$, there exists  a residual subset $X_0$ of $\widetilde{X}$ consisting of $m$-irregular vector for some $m\in \bbn$.
    For every $x\in X$ and $y\in x+X_0$, $(x,y)$ is Li-Yorke $m$-extremely chaotic.
    Note that $x$ is an accumulation point of $x+X_0$. Then $T$ is Li-Yorke $m$-extremely sensitive.
\end{proof}

By the proof of (1)$\Rightarrow$(6) in Theorem \ref{thm:f-space-LY-chaos}, we have the following.
\begin{coro}
    Let $T\in L(X)$. Then the set of $m$-irregular vectors for some $m\in\bbn$ is dense in the set of semi-irregular vectors.
\end{coro}

Following \cite{BBMP2015}, we say that $T\in L(X)$ satisfies the \emph{Li-Yorke chaos criterion} if there exists a subset $X_0$ of $X$ with the following properties:
\begin{enumerate}
    \item for every $x\in X_0$, $(T^nx)_n$ has a subsequence converging to $\vecz$;
    \item there is a bounded sequence $(a_n)_n$ in $\overline{\sspan(X_0)}$ such that the sequence $(T^na_n)_n$ is unbounded.
\end{enumerate}

Using the ideas developed in this subsection, we can give a new proof of the following result.
\begin{prop}[{\cite[Theorem 15]{BBMP2015}}] \label{prop:T-chaos-criterion}
    Let $T\in L(X)$. Then $T$ is Li-Yorke chaotic if and only if $T$ satisfies the Li-Yorke chaos criterion.
\end{prop}
\begin{proof}
    If $T$ is Li-Yorke chaotic, by Theorem~\ref{thm:f-space-LY-chaos} the restriction of $T$ to some closed $T$-invariant subspace $\widetilde{X}$ has a residual set of $m$-irregular points for some $m\in\bbn$.
    Let $X_0=\prox(T_{\widetilde{X}},\vecz)$.
    Then, by Theorem~\ref{thm:F-space-dense-LY-chaos} and Proposition~\ref{prop:f-space-sen}, $X_0$ is as required for the Li-Yorke chaos criterion.

    Now assume that $T$ satisfies the Li-Yorke chaos criterion. If $X_0\setminus  \asym(T,e)\neq\emptyset$, that is, if there exists a semi-irregular vector in $X_0$, then by Theorem~\ref{thm:f-space-LY-chaos} $T$ is Li-Yorke chaotic. Otherwise, $X_0\subset \asym(T,e)$.
    In this case, let $\widetilde{X}$ be the closure of $\asym(T,e)$. By Lemma~\ref{lem:subseq-asym-subspace}, $\widetilde{X}$ is a $T$-invariant closed subspace of $X$.
    Then, the proximal  cell of $e$ with respect to $T|_{\widetilde{X}}$ is dense in $\widetilde{X}$.
    By Proposition~\ref{prop:f-space-sen}, $T|_{\widetilde{X}}$ is sensitive.
    By applying Theorem~\ref{thm:F-space-dense-LY-chaos} to $T|_{\widetilde{X}}$, one has that
    $T$ is Li-Yorke chaotic.
\end{proof}

Let $T\in L(X)$. 
A vector subspace $Y$ of $X$ is called an \emph{irregular manifold} for $T$ if every vector $y\in Y\setminus\{0\}$ is irregular for $T$.

The following result is inspired by \cite[Theorem 20]{BBMP2015}, which states that if the asymptotic cell of $\vecz$ is dense in $X$ and $T$ is sensitive then $T$ admits a dense irregular manifold. The proof here is  different from that of \cite[Theorem 20]{BBMP2015} and the statement also contains that of \cite[Corollary 21]{BBMP2015}.

\begin{thm}\label{thm:dense-irregular-manifold}
    Let $X$ be a separable Fr\'echet space and $T\in L(X)$.
    If the asymptotic cell of $\vecz$ is dense in $X$,
    then either the asymptotic cell of $\vecz$ is $X$ or $T$ admits a dense irregular manifold.
\end{thm}
\begin{proof}
    Note that  $\prox(T,\vecz)$ is residual in $X$, because $\asym(T,\vecz)\subset \prox(T,\vecz)$.
    If $\asym(T,\vecz)= \prox(T,\vecz)$ then, by Lemma~\ref{lem:g-asym-prox-left-inv}\,(3), $\asym(T,\vecz)=X$.
    Otherwise $\prox(T,\vecz)\setminus \asym(T,\vecz)\neq\emptyset$, that is, $T$ admits a semi-irregular vector.
    Then, by Corollary~\ref{cor:scrambled-to-sensitive}, $T$ is sensitive.
    Combining Lemma~\ref{lem:ext-sen-equivalent} and Proposition~\ref{prop:f-space-sen},
    there exists $m\in\bbn$ such that the set
    \[
    D:=\Bigl\{x\in X\colon \limsup_{n\to\infty} p_m(T^nx)=\infty\Bigr\}
    \]
    is a dense $G_\delta$ subset of $X$. We need the following claim.
\smallskip 

    \noindent\textbf{Claim}: For any increasing sequence $(s_k)_k$ in $\bbn$, the set
    \[
    P(s_k)=\Bigl\{x\in X\colon \liminf_{k\to\infty} d(T^{s_k}x,\vecz)=0\Bigr\}
    \]
    is a dense $G_\delta$ subset of $X$.
    \begin{proof}[Proof of the Claim]
    Since $\asym(T,\vecz)\subset P(s_k)$, $P(s_k)$ is dense in $X$. Note that 
    \[
    P(s_k)=\bigcap_{n=1}^\infty
    \bigl\{x\in X\colon \exists k>n \text{ s.t. } d(T^{s_k}x,\vecz)<\tfrac{1}{n}\bigr\}.
    \]
    Then it is easy to see that $P(s_k)$ is a $G_\delta$ subset of $X$.
    \end{proof}
    Fix a dense sequence $(y_i)_i$ in $X$. We will construct recursively a sequence $(x_i)_i$ in $X$
    such that $d(x_i,y_i)<\frac{1}{i}$ and $\sspan\{x_i\colon i\in\bbn\}$ is an irregular manifold.

    By the completeness of $X$, $X_1:=P(k)\cap D$ is a dense $G_\delta$ subset of $X$. We pick $x_1\in X_1$ with $d(x_1,y_1)<1$.
    Then, there exist two increasing sequences $(s^{(1,1)}_k)_k$
    and $(t^{(1)}_k)_k$ in $\bbn$ such that
    \[
    \lim_{k\to\infty} d(T^{s^{(1,1)}_k}x_1,\vecz)=0
    \text{ and }
    \lim_{k\to\infty} p_m(T^{t^{(1)}_k}x_1)=\infty.
    \]
    It is clear that $\bbk x_1$ is an irregular manifold.

Assume that $x_i\in X$, sequences $(s_k^{(i,j)})_k$ and $(t_k^{(i)})_k$ in $\bbn$ have been constructed for $i=1,2,\dotsc,n$ and $j=1,\dotsc,i$ such that 
\begin{enumerate}
\item $d(x_i,y_i)<\frac{1}{i}$ for $i=1,2,\dotsc,n$;
\item for $i=2,\dotsc,n$ and $j=1,\dotsc,i-1$, $(s_k^{(i,j)})_k$ is a subsequence of $(s_k^{(i-1,j)})_k$,
and $(s_k^{(i,i)})_k$ is a subsequence of $(t_k^{(i-1)})_k$;
\item for $i=1,2,\dotsc,n$ and $j=1,\dotsc,i$, 
\[ 
\lim_{k\to\infty}d(T^{s_k^{(i,j)}}{x_i},\vecz)=0
\text{ and }
\lim_{k\to\infty}p_m(T^{t_k^{(i)}}{x_i})=\infty;
\]
\item $\sspan\{x_1,x_2,\dotsc,x_n\}$ is an irregular manifold.
\end{enumerate}
By the Claim, 
\[
X_{n+1}:=\bigcap_{j=1}^n P(s_k^{(n,j)}) \cap P(t_k^{(n)})\cap D
\]
is also a dense $G_\delta$ subset of $X$.
Now take $x_{n+1}\in X_{n+1}$ with $d(x_{n+1},y_{n+1})<\frac{1}{n+1}$.
For $j=1,\dotsc,n$, there exists a subsequence $(s_k^{(n+1,j)})_k$ of $(s_k^{(n,j)})_k$
such that 
\[ 
\lim_{k\to\infty}d(T^{s_k^{(n+1,j)}}{x_{n+1}},\vecz)=0,
\]
a  subsequence $(s_k^{(n+1,n+1)})_k$ of $(t_k^{(n)})_k$
such that 
\[ 
\lim_{k\to\infty}d(T^{s_k^{(n+1,n+1)}}{x_{n+1}},\vecz)=0,
\]
and an increasing sequence $(t_k^{(n+1)})_k$ in $\bbn$ such that
\[ 
\lim_{k\to\infty}p_m(T^{t_k^{(n+1)}}{x_{n+1}})=\infty.
\]
For any $\sum_{\ell=1}^{n+1} \alpha_\ell x_\ell\in  \sspan\{x_1,\dotsc,x_n,x_{n+1}\}\setminus\{\vecz\}$,
\begin{align*}
    \liminf_{k\to\infty} d\biggl(T^{s^{(n+1,1)}_k} \biggl(\sum_{\ell=1}^{n+1} \alpha_\ell x_\ell\biggl),\vecz\biggl)&
    \leq 
     \sum_{\ell=1}^{n+1}(|\alpha_\ell| +1)\lim_{k\to\infty}d(T^{s^{(n+1,1)}_k} x_{\ell},\vecz)
    =0.
\end{align*}
Let $i=\min\{\ell\in \{1,\dotsc,n+1\}\colon \alpha_\ell\neq 0\}$.
If $i<n+1$, then
\begin{align*}
    \limsup_{k\to\infty} p_m\biggl(T^{s^{(n+1,i+1)}_k}\biggl(\sum_{\ell=i}^{n+1} \alpha_\ell x_\ell\biggl)\bigg)
    &\geq 
     |\alpha_i| 
     \lim_{k\to\infty}p_m(T^{s^{(n+1,i+1)}_k} x_i)\\
    &\qquad -\sum_{\ell=i+1}^{n+1} |\alpha_\ell|\limsup_{k\to\infty} p_m(T^{s^{(n+1,i+1)}_k} x_\ell)
    \\
    &\geq 
     \infty 
    -\sum_{\ell=i+1}^{n+1} |\alpha_\ell|
    2^m\limsup_{k\to\infty} d(T^{s^{(n+1,i+1)}_k} x_\ell,\vecz)\\
    &=\infty-0=\infty.
\end{align*}
If $i=n+1$, then by the construction one has 
\[
\lim_{k\to\infty}p_m(T^{t_k^{(n+1)}}{x_{n+1}})=\infty.
\]
So $\sspan\{x_1,\dotsc,x_n,x_{n+1}\}$ is an irregular manifold. A recursive argument gives us a sequence $(x_i)_i$ fulfilling that the subspace $\sspan\{x_i\colon i\in\bbn\}$ is a dense irregular manifold.
\end{proof}

\section{Mean Li-Yorke chaos}\label{sec:MLYC}

\subsection{Mean equicontinuity, mean sensitivity and mean Li-Yorke chaos in topological dynamics}
\label{subsec:Meq-Msen-MLYC}
Let $(X, T)$ be a dynamical system with a metric $d$ on $X$.
A pair $(x,y)\in X\times X$ is called \emph{mean asymptotic}
if
\[
    \lim_{n\to\infty}\frac{1}{n}\sum_{i=1}^{n}d(T^ix,T^iy)=0,
\]
and \emph{mean proximal}
if
\[
    \liminf_{n\to\infty}\frac{1}{n}\sum_{i=1}^{n}d(T^ix,T^iy)=0.
\]
The \emph{mean asymptotic relation} and the \emph{mean proximal relation} of $(X, T)$, denoted by $\masym(T)$ and $\mprox(T)$, are the set of all mean asymptotic pairs and mean proximal pairs respectively.
For any $x\in X$, the \emph{mean asymptotic cell} and \emph{the mean proximal cell} of $x$ are defined by
\[
    \masym(T,x)=\{y\in X\colon (x,y)\in \masym(T)\}
\]
and
\[
    \mprox(T,x)=\{y\in X \colon (x,y) \in \mprox(T)\},
\]
respectively.

The following three lemmas are well-known, see e.g.~\cite{LTY2015,LY2016}.
Similar to Lemma~\ref{lem:dyn-prox-delta}, we have the following $G_\delta$ sets.

\begin{lem}\label{lem:mean-prox-G-delta}
    Let $(X, T)$ be a dynamical system and $\delta>0$.
    Then
    \begin{enumerate}
        \item for every $x\in X$, the mean proximal cell of $x$ is a $G_\delta$ subset of $X$;
        \item the mean proximal relation of $(X,T)$ is a $G_\delta$ subset of $X\times X$;
        \item for every $x\in X$,
              the set
              \[
                  \biggl\{y\in X\colon \limsup_{n\to\infty}\frac{1}{n}\sum_{i=1}^{n}d(T^ix,T^iy)\geq \delta\biggr\}
              \]
              is a $G_\delta$ subset of $ X$;
        \item the set
              \[
                  \biggl\{(x,y)\in X\times X\colon \limsup_{n\to\infty}\frac{1}{n}\sum_{i=1}^{n}d(T^ix,T^iy)\geq \delta\biggr\}
              \]
              is a  $G_\delta$ subset of $X\times X$.
    \end{enumerate}
\end{lem}
\begin{proof}
    (1). It follows from the following formula:
    \[
        \mprox(T,x)=\bigcap_{n=1}^{\infty}
        \biggl\{y\in X \colon \exists k>n \text{ s.t. } \frac{1}{k}\sum_{i=1}^{k}d(T^ix,T^iy)<\frac{1}{n}\biggr\}.
    \]

    The proofs of (2), (3) and (4) are similar to that of (1).
\end{proof}

Following~\cite{LTY2015}, we say that a dynamical system $(X,T)$ is \emph{mean equicontinuous}
if for any $\eps>0$ there exists some $\delta>0$ such that,
for any $x,y\in X$ with $d(x,y)<\delta$,
\[
    \limsup_{n\to\infty}\frac{1}{n}\sum_{i=1}^{n}d(T^ix,T^iy)<\eps.
\]

\begin{lem}\label{lem:mean-eq-prox-mean-asym}
    Let $(X, T)$ be a dynamical system.
    If $(X, T)$ is mean equicontinuous, then every proximal pair is mean asymptotic.
\end{lem}
\begin{proof}
    Let $(x_0,y_0)\in X\times X$ be a proximal pair.
    As $(X,T)$ is mean equicontinuous, for any $\eps>0$ there exists some
    $\delta>0$ such that, for any $x,y\in X$ with $d(x,y)<\delta$,
    \[
        \limsup_{n\to\infty}\frac{1}{n}\sum_{i=1}^{n}d(T^ix,T^iy)<\eps.
    \]
    As $(x_0,y_0)$ is proximal, there exists $k\in \bbn$ such that
    $d(T^kx_0,T^ky_0)<\delta$. Then
    \[
        \limsup_{n\to\infty}\frac{1}{n}\sum_{i=1}^{n}d(T^ix_0,T^iy_0)=\limsup_{n\to\infty}\frac{1}{n}\sum_{i=1}^{n}d(T^i(T^k x_0),T^i(T^ky_0))<\eps.
    \]
    By the arbitrariness of $\eps>0$, $(x_0,y_0)$ is mean asymptotic.
\end{proof}

Following~\cite{LTY2015}, we say that a dynamical system $(X,T)$ is \emph{mean sensitive}
if there exists some $\delta>0$ such that, for every $x\in X$ and $\eps>0$, 
there exists some $y\in X$ with $d(x,y)<\eps$ and
\[
    \limsup_{n\to\infty}\frac{1}{n}\sum_{i=1}^{n}d(T^ix,T^iy)\geq \delta.
\]

The proof of the following result is similar to that of Lemma~\ref{lem:sen-equivalent}, and we leave it to the reader.

\begin{lem}\label{lem:mean-sen-eq-cond}
    Let $(X, T)$ be a dynamical system with $X$ being completely metrizable.
    Then the following assertions are equivalent:
    \begin{enumerate}
        \item $(X,T)$ is mean sensitive;
        \item there exists $\delta>0$ such that for every $x\in X$, the set
              \[
                  \biggl\{y\in X\colon \limsup_{n\to\infty}  \frac{1}{n}\sum_{i=1}^{n}d(T^ix,T^iy) \geq \delta\biggr\}
              \]
              is a dense $G_\delta$ subset of $X$;
        \item there exists $\delta>0$ such that the set
              \[
                  \biggl\{(x,y)\in X\times X\colon  \limsup_{n\to\infty}\frac{1}{n}\sum_{i=1}^{n}d(T^ix,T^iy) \geq \delta\biggr\}
              \]
              is a dense $G_\delta$ subset of $X\times X$.
    \end{enumerate}
\end{lem}

Let $(X, T)$ be a dynamical system.
A pair $(x,y)\in X\times X$ is called \emph{mean Li-Yorke chaotic}
if
\[
    \liminf_{n\to\infty} \frac{1}{n}\sum_{i=1}^{n}d(T^ix,T^iy)=0\qquad \text{and}\qquad \limsup_{n\to\infty} \frac{1}{n}\sum_{i=1}^{n}d(T^ix,T^iy)>0,
\]
that is, $(x,y)$ is mean proximal but not mean asymptotic.
A subset $K$ of $X$ is called \emph{mean Li-Yorke scrambled} if any two distinct points $x,y\in K$ form a mean Li-Yorke chaotic pair.
We say that a dynamical system $(X,T)$ is \emph{mean Li-Yorke chaotic}
if there exists an uncountable mean Li-Yorke scrambled subset of $X$.

For a given positive number $\delta$, a pair $(x,y)\in X\times X$ is called \emph{mean Li-Yorke $\delta$-chaotic}
if
\[
    \liminf_{n\to\infty} \frac{1}{n}\sum_{i=1}^{n}d(T^ix,T^iy)=0\qquad \text{and}\qquad \limsup_{n\to\infty} \frac{1}{n}\sum_{i=1}^{n}d(T^ix,T^iy)\geq \delta.
\]
Similarly, we can define \emph{mean Li-Yorke $\delta$-scrambled set} and \emph{mean Li-Yorke $\delta$-chaos}.

We say that a dynamical system $(X,T)$ is \emph{mean Li-Yorke sensitive} if there exists some $\delta>0$ such that,
for every $x\in X$ and $\eps>0$, there exists some $y\in X$ with $d(x,y)<\eps$ such that $(x,y)$ is mean Li-Yorke $\delta$-chaotic.
It is clear that every mean Li-Yorke sensitive system is mean sensitive.

Similar to Proposition~\ref{prop:dense-delta-scrambled}, according to Lemmas~\ref{lem:mean-prox-G-delta} and~\ref{lem:mean-sen-eq-cond} we have the following result.

\begin{prop}\label{prop:dense--mean-delta-scrambled}
    Let $(X, T)$ be a dynamical system with $X$ being completely metrizable.
    Then the following assertions are equivalent:
    \begin{enumerate}
        \item there exists $\delta>0$ such that for every sequence $(O_j)_j$ of nonempty open subsets of $X$, there exists a sequence $(K_j)_j$ of Cantor sets with $K_j\subset O_j$ such that $\bigcup_{j=1}^\infty K_j$ is mean Li-Yorke $\delta$-scrambled;
        \item the mean proximal relation of $(X,T)$ is dense in $X\times X$ and $(X,T)$ is mean sensitive.
    \end{enumerate}
\end{prop}

Recall that every Fr\'echet space has a separating increasing sequence of seminorms, and that each of these seminorms can be seen as a pseudometric.
In order to apply the theory developed previously to the case of Fr\'echet spaces,
we introduce mean extreme sensitivity and mean Li-Yorke extreme chaos with respect to a continuous pseudometric.
From now on in this subsection, we fix a continuous pseudometric $\rho$ on $X$.
We say that a dynamical system $(X,T)$ is \emph{mean $\rho$-extremely sensitive} if  for every $x\in X$ and $\eps>0$ there exists $y\in X$ with $d(x,y)<\eps$ satisfying
\[
    \limsup_{n\to\infty} \frac{1}{n} \sum_{i=1}^n \rho(T^ix,T^iy)=\infty.
\]

The proof of the following result is straightforward.

\begin{lem}
    Let $(X, T)$ be a dynamical system with $X$ being completely metrizable.
    Then the following assertions are equivalent:
    \begin{enumerate}
        \item $(X,T)$ is mean $\rho$-extremely sensitive;
        \item for every $x\in X$, the set
              \[
                  \biggl\{y\in X\colon \limsup_{n\to\infty}  \frac{1}{n}\sum_{i=1}^{n}\rho(T^ix,T^iy) =\infty\biggr\}
              \]
              is a dense $G_\delta$ subset of $X$;
        \item the set
              \[
                  \biggl\{(x,y)\in X\times X\colon  \limsup_{n\to\infty}\frac{1}{n}\sum_{i=1}^{n}\rho(T^ix,T^iy) =\infty \biggr\}
              \]
              is a dense $G_\delta$ subset of $X\times X$.
    \end{enumerate}
\end{lem}

Let $(X, T)$ be a dynamical system.
A pair $(x,y)\in X\times X$ is called \emph{mean Li-Yorke $\rho$-extremely chaotic} if
\[
    \liminf_{n\to\infty} \frac{1}{n}\sum_{i=1}^{n}d(T^ix,T^iy)=0\qquad \text{and}\qquad \limsup_{n\to\infty} \frac{1}{n}\sum_{i=1}^{n}\rho(T^ix,T^iy)=\infty.
\]
Similarly, we can define \emph{mean Li-Yorke $\rho$-extremely scrambled set},  \emph{mean Li-Yorke $\rho$-extreme chaos} and
\emph{mean Li-Yorke $\rho$-extreme  sensitivity}.

Similar to Proposition~\ref{prop:dense--mean-delta-scrambled}, we have the following result.

\begin{prop}\label{prop:dense-mean-ext-scrambled}
    Let $(X, T)$ be a dynamical system with $X$ being completely metrizable.
    Then the following assertions are equivalent:
    \begin{enumerate}
        \item for every sequence $(O_j)_j$ of nonempty open subsets of $X$, there exists a sequence $(K_j)_j$ of Cantor sets with $K_j\subset O_j$ such that $\bigcup_{j=1}^\infty K_j$ is mean Li-Yorke $\rho$-extremely scrambled;
        \item the mean proximal relation of $(X,T)$ is dense in $X\times X$ and $(X,T)$ is mean $\rho$-extremely  sensitive.
    \end{enumerate}
\end{prop}

\subsection{Mean Li-Yorke chaos for continuous endomorphisms of completely metrizable groups}
\label{subsec:MLYC-group}
In \cite{BBP2020} the authors characterized mean Li-Yorke chaos for linear operators on Banach spaces.
In this subsection, we consider the more general setting of continuous endomorphisms of completely metrizable groups. We fix a completely metrizable group $G$ and a left-invariant compatible metric $d$ on $G$.

Let $T\in\eend(G)$. A point $x\in G$ is called \emph{mean semi-irregular} for $T$ if
\[
    \liminf_{n\to\infty} \frac{1}{n}\sum_{i=1}^{n} d(T^ix,e)=0\qquad \text{and}\qquad
    \limsup_{n\to\infty} \frac{1}{n}\sum_{i=1}^{n} d(T^ix,e)>0.
\]

The proof of the following result is similar to that of Lemma~\ref{lem:g-asym-prox-left-inv}.
\begin{lem}\label{lem:g-mean-asym-prox-left-inv}
    Let $T\in\eend(G)$.
    \begin{enumerate}
        \item For every $x\in G$, $\masym(T,x) = x \cdot \masym(T,e)$,
              $\mprox(T,x)=x\cdot \mprox(T,e)$.
        \item  $\masym(T,e)$ is dense in $ G$ if and only if $\masym(T)$ is dense in $G\times G$.
        \item If $\masym(T,e)$ is residual in $G$, then $\masym(T)=G$.
        \item $\mprox(T,e)$ is dense in $G$ if and only if
              $\mprox(T)$ is dense in $ G\times G$.
        \item For every $x,y\in G$, $(x,y)$ is a mean Li-Yorke chaotic pair if and only if $x^{-1}y$ is a mean semi-irregular point.
        \item If $K$ is a mean Li-Yorke scrambled subset of $G$,
              then for any $x\in G$, $xK$ is also a mean Li-Yorke scrambled set. In particular, for any $x\in K$, $x^{-1}K\setminus \{e\}$ consists of mean semi-irregular points.
        \item For a given $\delta>0$, if $K$ is a mean Li-Yorke $\delta$-scrambled subset of $G$,
              then for any  $x\in G$, $xK$ is also a mean Li-Yorke $\delta$-scrambled set.
    \end{enumerate}
\end{lem}

\begin{lem}\label{lem:subseq-mean-asym-subgroup}
    Let $T\in\eend(G)$.
    Then $\masym(T,e)$ is a $T$-invariant subgroup of $G$.
    In addition, if the metric $d$ is bounded or the endomorphism $T$ is Lipschitz continuous, then for any increasing sequence $(n_k)_{k}$ of positive integers,
    \[
        G(n_k):=\biggl\{x\in G\colon \lim_{k\to \infty} \frac{1}{n_k}\sum_{i=1}^{n_k} d(T^{i}x,e)=0\biggr\}
    \]
    is a $T$-invariant subgroup of $G$.
\end{lem}
\begin{proof}
    For every $x,y\in G(n_k)$, we have
    \begin{align*}
        \limsup_{k\to \infty} \frac{1}{n_k}\sum_{i=1}^{n_k} d(T^{i}(y^{-1}x),e)
         & =\limsup_{k\to \infty} \frac{1}{n_k}\sum_{i=1}^{n_k} d(T^{i}x,T^iy)                                                         \\
         & \leq\lim_{k\to\infty}\frac{1}{n_k}\sum_{i=1}^{n_k}d(T^{i}x,e)+\lim_{k\to\infty} \frac{1}{n_k}\sum_{i=1}^{n_k}d(e,T^{i}y)=0.
    \end{align*}
    Then $G(n_k)$ is a subgroup of $G$. Similarly, $\masym(T,e)$ is also a subgroup of $G$.

    For every $x\in \masym(T,e)$,
    \begin{align*}
        \lim_{n\to \infty} \frac{1}{n}\sum_{i=1}^{n} d(T^{i}(Tx),e)
         & = \lim_{n\to \infty} \frac{n+1}{n}\frac{1}{n+1} \biggl(\sum_{i=1}^{n+1} d(T^{i}x,e)- d(Tx,e)\biggr) \\
         & =  \lim_{n\to \infty} \frac{1}{n+1} \sum_{i=1}^{n+1} d(T^{i}x,e)=0.
    \end{align*}
    Then $Tx\in \masym(T,e)$, which implies that $\masym(T,e)$ is $T$-invariant.

    In addition, if $d$ is bounded by some constant $M>0$, then for every $x\in G(n_k)$,
    \begin{align*}
        \lim_{k\to \infty} \frac{1}{n_k}\sum_{i=1}^{n_k} d(T^{i}(Tx),e) & = \lim_{k\to \infty} \frac{1}{n_k}\biggl(\sum_{i=1}^{n_k} d(T^{i}x,e) +d(T^{n_k+1}x,e) -d(Tx,e)\biggr) \\
         & \leq \lim_{k\to \infty}\frac{1}{n_k} \biggl( \sum_{i=1}^{n_k}d(T^{i}x,e) +2M\biggr)=0.
    \end{align*}
    In this case, we have $Tx\in G(n_k)$.
    Then $G(n_k)$ is $T$-invariant.

 In addition, if $T$ is Lipschitz continuous, that is, if there exists $L>0$ such that $d(Tx,Ty)\leq Ld(x,y)$ for any $x,y\in G$, then for every $x\in G(n_k)$,
    \begin{align*}
        \lim_{k\to \infty} \frac{1}{n_k}\sum_{i=1}^{n_k} d(T^{i}(Tx),e) & = \lim_{k\to \infty} \frac{1}{n_k}\sum_{i=1}^{n_k} d(T(T^{i}x),Te)  \\
         & \leq L \lim_{k\to \infty}\frac{1}{n_k}\sum_{i=1}^{n_k} d(T^{i}x,e) =0.
    \end{align*}
    In this case, we also have $Tx\in G(n_k)$.
    Hence $G(n_k)$ is $T$-invariant.
\end{proof}

We have the following dichotomy of mean equicontinuity and mean sensitivity for continuous endomorphisms of completely metrizable groups.
\begin{thm}\label{thm:dich-mean-eq-mean-sen}
    Let $T\in\eend(G)$.
    Then $T$ is either mean equicontinuous or mean sensitive.
\end{thm}
\begin{proof}
    Let $e$ be the identity of $G$. We have the following two cases:

    Case 1: for any $\eps>0$, there exists  $\delta>0$ such that for every $x\in G$ with $d(x,e)<\delta$,
    \[
        \limsup_{n\to\infty}\frac{1}{n}\sum_{i=1}^{n}d(T^ix,e)<\eps.
    \]
    Let $x_1, x_2\in G$.
    If $d(x_1,x_2)<\delta$, by the left-invariance of $d$, we have $d(x_1,x_2)=d(x_2^{-1}x_1,e)<\delta$. Then
    \[
        \limsup_{n\to\infty}\frac{1}{n}\sum_{i=1}^{n}d(T^ix_1,T^ix_2)=\limsup_{n\to\infty}\frac{1}{n}\sum_{i=1}^{n}d(T^i(x_2^{-1}x_1),e)<\eps.
    \]
    This implies that $T$ is mean equicontinuous.

    Case 2: there exists $\delta>0$ such that for every $\eps>0$ one can find $y_\eps\in G$  with $d(y_\eps,e)<\eps$ such that
    \[
        \limsup_{n\to\infty}\frac{1}{n}\sum_{i=1}^{n}d(T^iy_\eps,e)\geq \delta.
    \]
    For every $x\in G$ and $\eps>0$,  we have
    $d(xy_\eps,x)<\eps$ and
    \[
        \limsup_{n\to\infty}\frac{1}{n}\sum_{i=1}^{n}d(T^i(xy_\eps), T^i x) =
        \limsup_{n\to\infty}\frac{1}{n}\sum_{i=1}^{n}d(T^i y_\eps, e) \geq \delta.
    \]
    Thus, $T$ is mean sensitive.
\end{proof}

Combining Lemma~\ref{lem:mean-eq-prox-mean-asym} and Theorem~\ref{thm:dich-mean-eq-mean-sen}, we have the following.

\begin{coro}\label{coro:g-mean-scrambled-sen}
    Let $T\in\eend(G)$. If there exists a mean Li-Yorke chaotic pair, then $T$ is mean sensitive.
\end{coro}

\begin{prop}\label{prop:g-mean-sen-eq-cond}
    Let $T\in \eend(G)$. Then the following assertions are equivalent:
    \begin{enumerate}
        \item $T$ is mean sensitive;
        \item there exists $\delta>0$ such that for every $x\in G$ and $\eps>0$
              there exists $y\in G$ with $d(x,y)<\eps$ and
              \[
                  \sup_{n\in\bbn} \frac{1}{n}\sum_{i=1}^{n}d(T^ix,T^iy)\geq \delta;
              \]
        \item there exists a sequence $(y_k)_k$ in $G$ and a sequence $(N_k)_k$ in $\bbn$ such that
              $\lim_{k\to\infty} y_k=e$ and
              \[
                 \inf_{k\in\bbn} \frac{1}{N_k} \sum_{i=1}^{N_k} d(T^i y_k,e)>0.
              \]
    \end{enumerate}
\end{prop}
\begin{proof}
    The implications (1)$\Rightarrow$(2)$\Rightarrow$(3) are clear.

    (3)$\Rightarrow$(1). Since $Te=e$, without loss of generality, we can assume the sequence $(N_k)_k$ to be increasing.
    Let 
    \[
        \delta=\inf_{k\in\bbn} \frac{1}{N_k} \sum_{i=1}^{N_k} d(T^i y_k,e).
    \]
    For each $n\in \bbn$, let
    \[
        X_n =\biggl\{x\in G\colon \exists k>n \text{ s.t. }
        \frac{1}{k} \sum_{i=1}^{k} d(T^ix,e)> \frac{\delta}{2}-\frac{1}{n}\biggr\}.
    \]
    Then each $X_n$ is an open subset of $G$. Let us show that it is also dense in $G$.
    Let $U$ be a nonempty open subset of $G$ and pick $x\in U$.
    If
    \[
        \limsup_{k\to\infty} \frac{1}{k}\sum_{i=1}^{k} d(T^ix,e)\geq \frac{\delta}{2},
    \]
    then $x\in X_n$. Now assume that
    \[
        \limsup_{k\to\infty} \frac{1}{k}\sum_{i=1}^{k} d(T^ix,e)< \frac{\delta}{2}.
    \]
    Then there exists $N\in\bbn$ such that for any $k\geq N$,
    \[
        \frac{1}{k}\sum_{i=1}^{k} d(T^ix,e) \leq \frac{\delta}{2}.
    \]
    Since $\lim_{k\to\infty} y_k x=x\in U$, there exists $k_0\in\bbn$ such that $y_{k_0}x\in U$ and $N_{k_0}>\max\{N,n\}$.
    Then
    \begin{align*}
        \frac{1}{N_{k_0}}\sum_{i=1}^{N_{k_0}}  d(T^i(y_{k_0}  x) ,e)
         & \geq \frac{1}{N_{k_0}}\sum_{i=1}^{N_{k_0}}
        d(T^iy_{k_0},e)
        - \frac{1}{N_{k_0}}\sum_{i=1}^{N_{k_0}} d(T^i(y_{k_0}  x),T^iy_{k_0}) \\
         & \geq \delta - \frac{1}{N_{k_0}}\sum_{i=1}^{N_{k_0}} d(T^i x,e)   \\
         & \geq \delta - \frac{\delta}{2}  = \frac{\delta}{2}.
    \end{align*}
    This implies that $y_{k_0}x\in X_n\cap U$.
    Therefore, the set
    \[
        \biggl\{x\in G\colon \limsup_{n\to\infty} \frac{1}{n}\sum_{i=1}^n d(T^i x,e)\geq \frac{\delta}{2}\biggr\}
        =\bigcap_{n=1}^\infty X_n
    \]
    is residual in $G$.
    By the proof of Theorem~\ref{thm:dich-mean-eq-mean-sen}, $T$ is mean sensitive.
\end{proof}

Combining Theorem~\ref{thm:dich-mean-eq-mean-sen} and Proposition~\ref{prop:g-mean-sen-eq-cond}, we have the following.
\begin{coro}\label{coro:g-mean-equi}
    Let $T\in\eend(G)$. Then $T$ is mean equicontinuous if and only if for any $\eps>0$ there exists some $\delta>0$ such that
    for every $x,y\in G$ with $d(x,y)<\delta$, 
    \[
        \sup_{n\in\bbn}\frac{1}{n}\sum_{i=1}^{n}d(T^ix,T^iy)<\eps.
    \]
\end{coro}

Now we give a characterization of the existence of a dense set of mean semi-irregular points for continuous endomorphisms of completely metrizable groups. Since the proof is similar to that of Theorem~\ref{thm:equi-of-dense-LY}, we leave it to the reader.

\begin{thm}\label{thm:equi-dense-mean-chaos}
    Let $T\in\eend(G)$.
    Then the following assertions are equivalent:
    \begin{enumerate}
        \item $T$ admits a dense set of mean semi-irregular points;
        \item $T$ admits a residual set of mean semi-irregular points;
        \item there exists $\delta>0$ such that for every sequence $(O_j)_j$ of nonempty open subsets of $G$, there exists a sequence $(K_j)_j$ of Cantor sets with $K_j\subset O_j$ such that $\bigcup_{j=1}^\infty K_j$ is mean Li-Yorke $\delta$-scrambled;
        \item the mean proximal relation of $T$ is dense in $G\times G$ and $T$ is mean sensitive;
        \item the mean proximal cell of $e$ is dense in $G$ and there exists $x\in G$ such that
              \[
                  \limsup_{n\to\infty} \frac{1}{n}\sum_{i=1}^{n} d(T^ix,e)>0.
              \]
    \end{enumerate}
\end{thm}

According to 
Theorem~\ref{thm:equi-dense-mean-chaos}\,(5), we have the following.
\begin{coro}
    Let $T\in\eend(G)$. If the mean proximal cell of $e$ is dense in $G$ and $T$ has a non-trivial periodic point, then it has a dense set of mean semi-irregular points.
\end{coro}

Similar to the proof of Theorem~\ref{thm:equi-of-LYC}, by applying Theorem~\ref{thm:equi-dense-mean-chaos} we also can characterize mean Li-Yorke chaos for continuous endomorphisms of completely metrizable groups.
Observe that, since we need Lemma~\ref{lem:subseq-mean-asym-subgroup}, the metric $d$ on $G$ should be bounded or the map $T$ should be Lipschitz continuous.

\begin{thm} \label{thm:equi-mean-chaos}
    Let $T\in\eend(G)$.
    If $d$ is  a bounded left-invariant compatible metric on $G$
    or $T$ is Lipschitz continuous,
    then the following assertions are equivalent:
    \begin{enumerate}
        \item $T$ admits a mean semi-irregular point;
        \item $T$ admits a mean Li-Yorke chaotic pair;
        \item $T$ is mean Li-Yorke chaotic;
        \item $T$ is mean Li-Yorke sensitive;
        \item the restriction of $T$ to some closed $T$-invariant subgroup $\widetilde{G}$ has a dense set of mean semi-irregular points.
    \end{enumerate}
\end{thm}

Similar to Proposition~\ref{prop:residul-scrambled-set}, we have the following result.
\begin{prop}
    Let $T\in\eend(G)$.
    Then the following assertions are equivalent:
    \begin{enumerate}
        \item $T$ admits a residual mean Li-Yorke scrambled set;
        \item every point in $G\setminus \{e\}$ is mean semi-irregular for $T$;
        \item $X$ is a mean Li-Yorke scrambled set for $T$.
    \end{enumerate}
\end{prop}

\begin{exam}\label{exam:SZ-mean-LY-chaos}
Let $\bbs(\bbz),\sigma, T, A$ as in Example~\ref{exam:SZ-LY-chaos}. 
As the mean proximal relation contains the asymptotic relation 
and $\sigma$ is a fixed point for $T$, 
according to Theorem~\ref{thm:equi-dense-mean-chaos}, $(\bbs(\bbz), T)$ is densely mean Li-Yorke chaotic for any left-invariant compatible metric on $\bbs(\bbz)$.
\end{exam}

\begin{exam}\label{exam:HR-mean-LY-chaos}
Let $H_+(\bbr),\sigma, T, A$ as in Example~\ref{exam:HR-LY-chaos}.
Similar to the idea in Example~\ref{exam:SZ-mean-LY-chaos}, one has $(H_+(\bbr), T)$ is densely mean Li-Yorke chaotic for any left-invariant compatible metric on $H_+(\bbr)$.
\end{exam}

\subsection{Mean Li-Yorke extreme chaos for continuous endomorphisms of completely metrizable groups}
\label{subsec:MLYeC-group}
With the same notation as in the previous subsection, we fix a completely metrizable group $G$ and a left-invariant compatible metric $d$ on $G$.
In order to study the case of Fréchet spaces in the next section, in which we will consider the seminorms on such a space, in this subsection we consider mean Li-Yorke chaos with respect to a left-invariant continuous pseudometric for metrizable groups.

In this subsection, we also fix a left-invariant continuous pseudometric $\rho$ on $G$.
Let $T\in\eend(G)$. A point $x\in G$ is called \emph{mean irregular} for $T$ if
\[
    \liminf_{n\to\infty} \frac{1}{n}\sum_{i=1}^{n} d(T^ix,e)=0\qquad \text{and}\qquad
    \limsup_{n\to\infty} \frac{1}{n}\sum_{i=1}^{n} \rho(T^ix,e)=\infty.
\]
The proofs of the following three results  are similar to those of Lemma~\ref{lem:g-asym-prox-left-inv} and Propositions~\ref{prop:irregular-ext-sen} and~\ref{prop:g-mean-sen-eq-cond} respectively.

\begin{lem}\label{lem:g-ext-scrambled-irr}
    Let $T\in\eend(G)$.
    \begin{enumerate}
        \item For every $x,y\in G$, $(x,y)$ is a mean Li-Yorke extremely chaotic pair if and only if $x^{-1}y$ is a mean irregular point.
        \item If $K$ is a mean Li-Yorke extremely scrambled subset of $G$,
              then for any $x\in G$, $xK$ is also a mean Li-Yorke extremely scrambled set. In particular, for any $x\in K$, $x^{-1}K\setminus \{e\}$ consists of mean  irregular points.
    \end{enumerate}
\end{lem}

\begin{prop}\label{prop:g-mean-ext-sen-eq-cond}
    Let $T\in \eend(G)$. Then the following assertions are equivalent:
    \begin{enumerate}
        \item $T$ is mean extremely sensitive;
        \item for any $x\in G$ and $\eps>0$
              there exists $y\in G$ with $d(x,y)<\eps$ and
              \[
                  \sup_{n\in\bbn} \frac{1}{n}\sum_{i=1}^{n}\rho(T^ix,T^iy)=\infty;
              \]
        \item there exists a sequence $(y_k)_k$ in $G$ and a sequence $(N_k)_k$ in $\bbn$ such that
              $\lim_{k\to\infty} y_k=e$ and
              \[
                 \frac{1}{N_k} \sum_{i=1}^{N_k} \rho(T^i y_k,e)\geq k,\ \forall k\in\bbn.
              \]
    \end{enumerate}
\end{prop}

\begin{prop}\label{prop:mean-irr-ext-sen}
    Let $T\in \eend(G)$. If there exists a mean irregular point, then $T$ is mean extremely sensitive.
\end{prop}

Similar to Theorem~\ref{thm:equi-dense-mean-chaos}, we give a characterization of the existence of a dense set of mean irregular points for continuous endomorphisms of completely metrizable groups. 

\begin{thm}\label{thm:equi-dense-mean-irr}
    Let $T\in \eend(G)$. Then the following assertions are equivalent:
    \begin{enumerate}
        \item $T$ admits a dense set of mean irregular points;
        \item $T$ admits a residual set of mean irregular points;
        \item  for every sequence $(O_j)_j$ of nonempty open subsets of $G$, there exists a sequence $(K_j)_j$ of Cantor sets with $K_j\subset O_j$ such that $\bigcup_{j=1}^\infty K_j$ is mean Li-Yorke extremely scrambled;
        \item the mean proximal relation of $T$ is dense in $G\times G$ and $T$ is mean extremely sensitive.
    \end{enumerate}
\end{thm}

By applying Theorem~\ref{thm:equi-dense-mean-irr}, we also can characterize mean Li-Yorke extreme chaos for continuous endomorphisms of completely metrizable groups.

\begin{thm} \label{thm:equi-mean-ext-chaos}
    Let $T\in\eend(G)$.
    If $d$ is  a bounded left-invariant compatible metric on $G$
    or $T$ is Lipschitz continuous,
    then the following assertions are equivalent:
    \begin{enumerate}
        \item $T$ admits a mean  irregular point;
        \item $T$ admits a mean Li-Yorke extremely chaotic pair;
        \item $T$ is mean Li-Yorke extremely chaotic;
        \item $T$ is mean Li-Yorke extremely sensitive;
        \item the restriction of $T$ to some closed $T$-invariant subgroup $\widetilde{G}$ has a dense set of mean irregular points.
    \end{enumerate}
\end{thm}

\subsection{Mean Li-Yorke chaos for continuous linear operators on Fr\'echet spaces}
\label{subsec:MLYC-Frechet}
In this subsection, unless otherwise specified, $X$ denotes an arbitrary infinite-dimensional Fr{\'e}chet space. The separating increasing sequence $(p_j)_j$ of seminorms on $X$ and the Fr{\'e}chet-space metric $d$ on $X$ are as in Subsection~\ref{subsec:LYC-Frechet}.
Even though mean Li-Yorke chaos for linear operators on Banach spaces has been characterized in \cite{BBP2020}, in this subsection we generalize the results to the setting of linear operators on Fr\'echet spaces by applying the results from Subsections~\ref{subsec:MLYC-group} and~\ref{subsec:MLYeC-group} to continuous linear operators on $X$.
We will also obtain some new results utilizing the linear structure of $X$.

Since the Fr{\'e}chet-space metric $d$ on $X$ is bounded and $d(\lambda x,\vecz)\leq (1+|\lambda|)d(x,\vecz)$ for any $\lambda\in\bbk$ and $x\in X$, following Lemmas~\ref{lem:subseq-asym-subspace} and~\ref{lem:subseq-mean-asym-subgroup}, we have the following result.
\begin{lem}\label{lem:seq-mean-asym-subspace}
    Let $T\in L(X)$ and
    $(n_k)_{k}$ be an increasing sequence of positive integers.
    Then
    \[
        \biggl\{x\in X\colon \lim_{k\to \infty} \frac{1}{n_k}\sum_{i=1}^{n_k} d(T^{i}x,\vecz)=0\biggr\}
    \]
    is a $T$-invariant subspace of $X$.
\end{lem}

With the same idea as in Subsection~\ref{subsec:LYC-Frechet}, 
we say that $T\in L(X)$ is \emph{mean $m$-extremely sensitive} if it is mean extremely sensitive with respect to the seminorm $p_m$.

\begin{prop}\label{prop:p-m-infty}
    Let $T\in L(X)$ and $m\in\bbn$. Then the following assertions are equivalent:
    \begin{enumerate}
        \item $T$ is mean $m$-extremely sensitive;
        \item there exists  a vector $x\in X$ such that
              \[
                  \limsup_{n\to\infty} \frac{1}{n}\sum_{i=1}^n p_m( T^ix)=\infty;
              \]
        \item there exists a bounded sequence $(y_k)_k$ in $X$ and an increasing sequence $(N_k)_k$ in $\bbn$ such that
              \[
                  \sup_{k\in\bbn} \frac{1}{N_k} \sum_{i=1}^{N_k} p_m(T^i y_k)=\infty;
              \]
        \item there exists  a  sequence $(y_k)_k$ in $X$ and an increasing sequence $(N_k)_k$ in $\bbn$ such that $\lim_{k\to\infty}y_k=\vecz$ and
              \[
                  \inf_{k\in\bbn} \frac{1}{N_k} \sum_{i=1}^{N_k} p_m(T^i y_k)>0.
              \]
    \end{enumerate}
\end{prop}

\begin{proof}
    The implications (1)$\Rightarrow$(2)$\Rightarrow$(3) are clear.

    (3)$\Rightarrow$(4). For each $k\in\bbn$, let $a_k= \frac{1}{N_k} \sum\limits_{i=1}^{N_k} p_m(T^i y_k)$. Without loss of generality, we can assume that $(a_k)_k$ is increasing. For each $k\in\bbn$, let $x_k=\frac{y_k}{a_k}$.
    Then $\lim_{k\to\infty}x_k=\vecz$ and
    \[
        \inf_{k} \frac{1}{N_k} \sum_{i=1}^{N_k} p_m(T^i x_k)\geq 1.
    \]

    (4)$\Rightarrow$(1).
Let  
\[ 
\delta= \inf_{k\in\bbn} \frac{1}{N_k} \sum_{i=1}^{N_k} p_m(T^i y_k).
\]
    For each $n\in\bbn$, let
    \[
        X_n=\biggl\{x\in X\colon \exists k>n \text{ s.t. } \frac{1}{k} \sum_{i=1}^k p_m(T^ix)>n\biggr\}.
    \]
    It is clear that $X_n$ is open.
    Let us show that $X_n$ is dense.
    Let $U$ be a nonempty open subset of $X$ and pick $y\in X$.
    If
    \[
        \limsup_{k\to\infty} \frac{1}{k}\sum_{i=1}^k  p_m(T^i y)=\infty,
    \]
    then $y\in X_n$.
    Now assume that
    \[
        \limsup_{k\to\infty} \frac{1}{k}\sum_{i=1}^k  p_m(T^i y)=M<\infty.
    \]
    Then there exists $N\in\bbn$ such that for any $k\geq N$, we have
    \[
        \frac{1}{k}\sum_{i=1}^k  p_m(T^i y)\leq 2M.
    \]
    There exists $k_0\in\bbn$ such that $y+\frac{2M+2n}{\delta} y_{k_0}\in U$ and $N_{k_0}>\max\{N,n\}$.
    Then
    \begin{align*}
        \frac{1}{N_{k_0}}\sum_{i=1}^{N_{k_0}}  p_m(T^i(y + \tfrac{M+2n}{\delta} y_{k_0}))
         & \geq \frac{1}{N_{k_0}}\sum_{i=1}^{N_{k_0}}
        p_m(T^i( \tfrac{2M+2n}{\delta}y_{k_0}))
        - \frac{1}{N_{k_0}}\sum_{i=1}^{N_{k_0}}  p_m(T^iy) \\
         & \geq \tfrac{2M+2n}{\delta} \cdot\delta -2M >n
    \end{align*}
    This implies that $y+\frac{2M+2n}{\delta} x_{j_0}\in X_n$.
    Therefore, the set
    \[
        \biggl\{x\in X\colon \limsup_{n\to\infty} \frac{1}{n}\sum_{i=1}^n  p_m(T^i x)=\infty\biggr\}
        =\bigcap_{n=1}^\infty X_n
    \]
    is residual in $X$, which implies that $T$ is mean $m$-extremely sensitive.
\end{proof}

\begin{prop}\label{prop:f-mean-sen-2-mean-ext-sen}
    Let $T\in L(X)$. If $T$ is mean sensitive then $T$ is mean $m$-extremely sensitive for some $m\in\bbn$.
\end{prop}
\begin{proof}
By Proposition~\ref{prop:g-mean-sen-eq-cond}\,(3), there exist sequences $(y_k)_k$ and $(N_k)_k$ in $\bbn$ such that $\lim_{k\to \infty}y_k=0$ and also
\[
    \delta:=\inf_{k\in \bbn}\frac{1}{N_k}\sum_{i=1}^{N_k}d(T^i y_k,\vecz)>0.
\]
Since for every $x\in X$ and every $m\in\bbn$ we have that $d(x,0)\leq p_m(x)+\frac{1}{2^m}$, we deduce
\[
    \delta=\inf_{k\in \bbn}\frac{1}{N_k}\sum_{i=1}^{N_k}d(T^i y_k,\vecz)\leq \inf_{k\in \bbn}\frac{1}{N_k}\sum_{i=1}^{N_k}(p_m (T^i y_k)+\frac{1}{2^m})
    \leq \frac{1}{2^m}+\inf_{k\in \bbn}\frac{1}{N_k}\sum_{i=1}^{N_k}p_{m}(T^i y_k),
\]
for every $m\in\bbn$. If now we take $m\in \bbn$ such that $\delta>\frac{1}{2^m}$, we obtain
\[
    0<\delta-\frac{1}{2^m}\leq \inf_{k\in \bbn}\frac{1}{N_k}\sum_{i=1}^{N_k}p_{m}(T^i y_k).
\]
By Proposition~\ref{prop:p-m-infty}\,(4), it follows that $T$ is mean $m$-extremely sensitive.
\end{proof}

In \cite{BBP2020}, the authors introduced the concepts of \emph{absolutely mean semi-irregular vector} and \emph{absolutely mean irregular vector}
for linear operators on Banach spaces. These notions can be naturally extended to the Fréchet space case with respect to the seminorm $p_m$. 
Following as in Subsection~\ref{subsec:LYC-Frechet}, we can define  \emph{mean Li-Yorke $m$-extremely scrambled set}, \emph{mean Li-Yorke $m$-extreme chaos} and \emph{mean Li-Yorke $m$-extreme sensitivity}.

Now we give a characterization of the existence of a  dense set of absolutely mean $m$-irregular vectors for continuous linear operators on Fr\'echet spaces. Since the proof is similar to that of Theorem~\ref{thm:F-space-dense-LY-chaos}, we leave it to the reader.

\begin{thm}\label{thm:f-dense-mean-extreme-chaos}
Let $T\in L(X)$. Then we have the implications 
    (1) $\Leftrightarrow$ (2) $\Rightarrow$ (3) $\Leftrightarrow$ (4) $\Leftrightarrow$ (5), where
    \begin{enumerate}
        \item $T$ admits a dense set of absolutely mean semi-irregular vectors;
        \item the mean proximal  cell of $\vecz$ is dense in $X$ and there exists $x\in X$ such that 
         \[
        \limsup_{n\to\infty} \frac{1}{n}\sum_{i=1}^nd(T^ix,\vecz)>0;
         \]
        \item $T$ admits a residual set of absolutely mean $m$-irregular vectors for some $m\in\bbn$;
        \item there exists $m\in\bbn$ such that for every sequence $(O_j)_j$ of nonempty open subsets of $X$, there exists a sequence $(K_j)_j$ of Cantor sets with $K_j\subset O_j$ such that $\bigcup_{j=1}^\infty K_j$ is mean Li-Yorke $m$-extremely scrambled;
        \item the mean proximal relation of $T$ is dense in $X\times X$ and $T$ is mean  $m$-extremely sensitive for some $m\in\bbn$.
    \end{enumerate}
\end{thm}


Similar to the proof of Theorem~\ref{thm:f-space-LY-chaos}, by applying Theorem~\ref{thm:f-dense-mean-extreme-chaos} we get the following characterization of mean Li-Yorke $m$-extreme chaos for continuous linear operators on Fr\'echet spaces.

\begin{thm}
    Let $T\in L(X)$ and $m\in\bbn$. Then the following assertions are equivalent:
    \begin{enumerate}
        \item $T$ admits an absolutely mean $m$-irregular vector;
        \item $T$ admits a mean Li-Yorke $m$-extremely chaotic pair;
        \item $T$ is mean Li-Yorke $m$-extremely chaotic;
        \item $T$ is mean Li-Yorke $m$-extremely sensitive;
        \item the restriction of $T$ to some closed $T$-invariant subspace $\widetilde{X}$ has a dense set of  absolutely mean $m$-irregular vectors.
    \end{enumerate}
\end{thm}

 The following result reveals that any absolutely mean semi-irregular vector is an accumulation point of the set of absolutely mean $m$-irregular vectors for some $m\in\bbn$.

\begin{prop}
    Let $T\in L(X)$. Then the set of absolutely mean semi-irregular vectors is contained in the closure of the set of absolutely mean $m$-irregular vectors for some $m\in\bbn$.
\end{prop}
\begin{proof}
    Let $y$ be an absolutely mean semi-irregular vector.
    There exists  an increasing sequence $(n_k)_{k}$ of positive integers such that
    \[
        \lim_{k\to \infty} \frac{1}{n_k}\sum_{i=1}^{n_k} d(T^{i}y,\vecz)=0.
    \]
    By Lemma~\ref{lem:seq-mean-asym-subspace}, we have that
    \[
        X_0:=\biggl\{x\in X\colon \lim_{k\to \infty} \frac{1}{n_k}\sum_{i=1}^{n_k} d(T^{i}x,\vecz)=0\biggr\}
    \]
    is a $T$-invariant subspace of $X$.
    Let $\widetilde{X}$ be the closure of $X_0$.
    Then $\widetilde{X}$ is a $T$-invariant closed subspace of $X$.
    As the mean proximal cell of $e$ contains $X_0$, the mean proximal relation of $T|_ {\widetilde{X}}$ is dense in  $\widetilde{X}\times  \widetilde{X}$.
    Since $y\in  \widetilde{X}$ and $(y,e)$ is a mean Li-Yorke chaotic  pair, then by Corollary~\ref{coro:g-mean-scrambled-sen}, $T|_ {\widetilde{X}}$ is mean sensitive. By Proposition~\ref{prop:f-mean-sen-2-mean-ext-sen}, we have that $T|_{\widetilde{X}}$ is mean $m$-extremely sensitive for some $m\in\bbn$, so that by Proposition~\ref{prop:p-m-infty} there exists a vector $x\in \widetilde{X}$ with
    \[
        \limsup_{n\to\infty} \frac{1}{n}\sum_{i=1}^n p_m(T^i x)=\infty.
    \]
    By applying Theorem~\ref{thm:f-dense-mean-extreme-chaos} to $T|_ {\widetilde{X}}$ and $m$, the set of absolutely mean $m$-irregular vectors is residual in $\widetilde{X}$.
    Then $y$ is an accumulation point of the set of absolutely mean $m$-irregular vectors.
\end{proof}

Let $T\in L(X)$ and $m\in\bbn$. A vector subspace $Y$ of $X$ is called a \emph{absolutely mean $m$-irregular manifold} for $T$ if every vector $y\in Y\setminus\{0\}$ is absolutely mean $m$-irregular for $T$. 
The following result was proved in \cite[Theorem 29]{BBP2020} for the case in which $X$ is a Banach space.
Since the proof is similar to that of  Theorem~\ref{thm:dense-irregular-manifold}, we leave it to the reader.

\begin{prop} \label{prop:dense-mean-irregular-manifold}
    Let $X$ be a separable Fr\'echet space and $T\in L(X)$. If the mean asymptotic cell of $\vecz$ is dense in $X$  
    then either the mean asymptotic cell of $\vecz$ is  $X$ or $T$ admits a dense absolutely mean $m$-irregular manifold for some $m\in\bbn$.
\end{prop}

\subsection{Mean Li-Yorke chaos for continuous linear operators on Banach spaces}
\label{subsec:MLYC-Banach}
In this subsection, we consider continuous linear operators on Banach spaces.
Let $(X,\lVert\cdot\rVert)$ be a Banach space.
The Banach-space metric induced by the norm is
\[
    d_B(x,y)=\lVert x-y\rVert, \quad \forall x,y\in X.
\]
For a continuous linear operator $T$ on a Banach space $X$, unless otherwise specified we only consider the dynamical properties with respect to the metric $d_B$ or simply the norm $\lVert\cdot\rVert$. This section aims to reprove the characterizations of Li-Yorke chaos for linear operators on Banach spaces from \cite{BBP2020} via our new rather general topological dynamics approach.

Since every continuous linear operator of a Banach space is Lipschitz continuous, following Lemmas~\ref{lem:subseq-mean-asym-subgroup} and~\ref{lem:seq-mean-asym-subspace}, we have the following result.
\begin{lem}\label{lem:B-seq-mean-asym-subspace}
    Let $(X,\lVert\cdot\rVert)$ be a Banach space, $T\in L(X)$ and
    $(n_k)_{k}$ be an increasing sequence of positive integers.
    Then
    \[
        X_0:=\biggl\{x\in X\colon \lim_{k\to \infty} \frac{1}{n_k}\sum_{i=1}^{n_k} \lVert T^{i}x\rVert =0\biggr\}
    \]
    is a $T$-invariant subspace of $X$.
\end{lem}

Following \cite{LH2015}, we say that $T$ is \emph{absolutely Ces\`aro bounded} if there exists a constant $C>0$ such that
\[
    \sup_{n\in\bbn}\frac{1}{n}\sum_{i=1}^n\Vert T^i x\Vert \leq C \Vert x\Vert
\]
for all $x\in X$.
We have the following characterization of mean equicontinuity for continuous linear operators on Banach spaces, which can be regarded as a mean version of the Banach-Steinhaus Theorem.

\begin{thm}\label{thm:B-mean-eq}
    Let $(X,\lVert\cdot\rVert)$ be a Banach space and $T\in L(X)$.
    Then the following assertions are  equivalent:
    \begin{enumerate}
        \item $T$ is mean equicontinuous;
        \item  for any $\eps>0$ there exists some $\delta>0$ such that,
              for every $x\in X$ with $\lVert x\lVert <\delta$,
              \[
                  \sup_{n\in\bbn}\frac{1}{n}\sum_{i=1}^{n}\lVert T^ix\lVert<\eps;
              \]
        \item $T$ is absolutely Ces\`aro bounded;
        \item for every $x\in X$,
              \[
                  \sup_{n\in\bbn}\frac{1}{n}\sum_{i=1}^{n}\lVert T^ix\lVert<\infty.
              \]
    \end{enumerate}
\end{thm}
\begin{proof}
    (1)$\Leftrightarrow$(2). It follows from Corollary~\ref{coro:g-mean-equi}.

    (2)$\Rightarrow$(3). There exists $\delta_1>0$ such that for every $ x\in X$ with $\lVert x\lVert <\delta_1$, we have
    \[
        \sup_{n\in\bbn}\frac{1}{n}\sum_{i=1}^{n}\lVert T^ix\lVert<1.
    \]
    Let $C= \frac{2}{\delta_1} $.
    For every $x\in X$, we have $ \bigl\lVert \frac{\delta_1 x}{2\lVert x\rVert}\bigr\rVert<\delta_1$, and then
    \[
        \sup_{n\in\bbn}\frac{1}{n}\sum_{i=1}^{n}\lVert T^ix\lVert
        = \frac{2\lVert x\rVert}{\delta_1} \sup_{n\in\bbn}\frac{1}{n}\sum_{i=1}^{n}\Bigl\lVert T^i\Bigl(\frac{\delta_1 x}{2\lVert x\rVert}\Bigr)\Bigr\rVert< C\lVert x\rVert.
    \]
    This shows that  $T$ is absolutely Ces\`aro bounded.

    The implication (3)$\Rightarrow$(4) is clear.

    (4)$\Rightarrow$(2).  For each $N\in\bbn$, let
    \[
        X_N = \biggl\{x\in X\colon \sup_{n\in\bbn}\frac{1}{n}\sum_{i=1}^{n}\lVert T^ix\lVert\leq N\biggr\}.
    \]
    It is clear that $X_n$ is closed and $X=\bigcup_{n=1}^\infty X_n$.
    By the Baire category theorem, there exists $N_0\in\bbn$ such that the interior of $X_{N_0}$ is not empty.
    Then $X_{N_0}-X_{N_0}$ is a neighborhood of $\vecz$, that is, there exists $\delta_0>0$ such that $\{x\in X\colon \lVert x\rVert <\delta_0\}\subset X_{N_0}-X_{N_0}$.
    For any $x\in X$ with $\lVert x\rVert <\delta_0$, there exists $x_1,x_2\in X_{N_0}$ such that $x=x_1-x_2$ and then
    \begin{align*}
        \sup_{n\in\bbn}\frac{1}{n}\sum_{i=1}^{n}\lVert T^ix\lVert
         & = \sup_{n\in\bbn}\frac{1}{n}\sum_{i=1}^{n}\lVert T^i(x_1-x_2)\lVert                                                               \\
         & \leq \sup_{n\in\bbn}\frac{1}{n}\sum_{i=1}^{n}\lVert T^i(x_1)\lVert+ \sup_{n\in\bbn}\frac{1}{n}\sum_{i=1}^{n}\lVert T^i(x_2)\lVert \\
         & \leq N_{0}+N_{0}=2N_0.
    \end{align*}
    Thus (2) follows from the linearity of $T$.
\end{proof}

Combining Propositions~\ref{prop:g-mean-sen-eq-cond} and \ref{prop:p-m-infty} and Theorem~\ref{thm:B-mean-eq}, we have the following equivalent conditions for mean sensitive operators on Banach spaces.
\begin{prop}\label{prop:B-mean-sen}
    Let $X$ be a Banach space and $T\in L(X)$. Then the following assertions are equivalent:
    \begin{enumerate}
        \item $T$ is mean sensitive;
        \item $T$ is not absolutely Ces\`aro bounded;
        \item $T$ is mean extremely sensitive.
    \end{enumerate}
\end{prop}

Now we give a characterization of the existence of a dense set of mean irregular vectors for continuous linear operators on Banach spaces, which slightly strengthens \cite[Theorems 17 and 22]{BBP2020}.
Since the proof is similar to that of Theorem~\ref{thm:F-space-dense-LY-chaos}, we leave it to the reader.

\begin{thm}\label{thm:B-dense-mean-extreme-chaos}
    Let $(X,\lVert\cdot\rVert)$ be a Banach space and $T\in L(X)$. Then the following assertions are equivalent:
    \begin{enumerate}
        \item $T$ admits a dense set of absolutely mean semi-irregular vectors;
        \item $T$ admits a residual set of absolutely mean irregular vectors;
        \item for every sequence $(O_j)_j$ of nonempty open subsets of $X$, there exists a sequence $(K_j)_j$ of Cantor sets with $K_j\subset O_j$ such that $\bigcup_{j=1}^\infty K_j$ is mean Li-Yorke extremely scrambled;
        \item the mean proximal relation of $T$ is dense in $X\times X$ and $T$ is mean sensitive;
        \item the mean proximal cell of $\vecz$ is dense in $X$ and there exists $x\in X$ such that
              \[
                  \limsup_{n\to\infty}\frac{1}{n}\sum_{i=1}^n \lVert T^ix\rVert >0.
              \]
    \end{enumerate}
\end{thm}

According to 
Theorem~\ref{thm:B-dense-mean-extreme-chaos}\,(5), we have the following.
\begin{coro}
    Let $(X,\lVert\cdot\rVert)$ be a Banach space and $T\in L(X)$. Assume that the mean proximal cell of $\vecz$ is dense in $X$.
    If $T$ has a non-trivial periodic point or an eigenvalue $\lambda$ with $|\lambda|\geq 1$, then it has a residual set of absolutely mean irregular vectors.
\end{coro}

The following result characterizes mean Li-Yorke chaos for continuous linear operators on Banach spaces, which is essentially contained in 
\cite[Theorem~5]{BBP2020}. Here we can apply Thereon~\ref{thm:B-dense-mean-extreme-chaos} to get a direct proof.

\begin{thm}\label{thm:B-mean-LY-chaos}
    Let $(X,\lVert\cdot\rVert)$ be a Banach space and $T\in L(X)$. Then the following assertions are equivalent:
    \begin{enumerate}
        \item $T$ admits an absolutely mean semi-irregular vector;
        \item $T$ admits a  mean Li-Yorke chaotic pair;
        \item $T$ admits an absolutely mean irregular vector;
        \item $T$ is mean Li-Yorke extremely chaotic;
        \item $T$ is mean Li-Yorke extremely sensitive;
        \item the restriction of $T$ to some closed $T$-invariant subspace $\widetilde{X}$ has a residual set of absolutely mean irregular vectors.
    \end{enumerate}
\end{thm}

\begin{coro}[{\cite[Theorem~7]{BBP2020}}]
    Let $(X,\lVert\cdot\rVert)$ be a Banach space and $T\in L(X)$.
    Then the set of absolutely mean irregular vectors is dense in the set of absolutely mean semi-irregular vectors.
\end{coro}

Following \cite{BBP2020}, we say that $T\in L(X)$ satisfies the \emph{mean Li-Yorke chaos criterion}
if there exists a subset $X_0$ of $X$ with the following properties:
\begin{enumerate}
    \item for every $x\in X_0$, $\liminf\limits_{n\to\infty}\frac{1}{n}\sum\limits_{i=1}^n \lVert T^ix\rVert =0$;
    \item  there exists a sequence $(y_k)_k$ in $\overline{\sspan(X_0)}$ and a sequence $(N_k)_k$ in $\bbn$ such that for every $k\in\bbn$, we have 
          \[
              \frac{1}{N_k}\sum_{i=1}^{N_k}\lVert T^iy_k\rVert\geq k\lVert y_k\rVert.
          \]
\end{enumerate}

Using the ideas developed in this subsection, we can give a new proof for the following result.
Since the proof is similar to that of Proposition~\ref{prop:T-chaos-criterion}, we leave it to the reader.

\begin{prop}[{\cite[Theorem 9]{BBP2020}}]
    Let $(X,\lVert\cdot\rVert)$ be a Banach space and $T\in L(X)$.
    Then $T$ is mean Li-Yorke chaotic if and only if it satisfies the mean Li-Yorke chaos criterion.
\end{prop}

\begin{rem}\label{rem:m-of-BF-BB}
    Let $(X,\lVert\cdot\rVert)$ be a Banach space and $T\in L(X)$.
    As a Fr\'echet space, $X$ has the Fr\'echet-space metric
    \[
        d_F(x,y)=\min(\lVert x-y\rVert,1), \quad \forall x,y\in X.
    \]
   
    It should be noticed that mean Li-Yorke chaos with respect to the metric $d_B$ may be different from the one with respect to the metric $d_F$.

    In \cite[Theorem 25]{BBPW2018}, the authors constructed a continuous linear operator $T$ on $X=c_0(\bbn)$ or $X=\ell^p(\bbn)$ for some $1\leq p<\infty$ such that $T$ is distributionally chaotic of type $1$ (see the next section for the definition) and for all $ x\in X\setminus\{0\}$,
    \[
        \lim_{n\to\infty}\frac{1}{n}\sum_{i=1}^n\lVert T^ix\rVert =\infty.
    \]
    Then $T$ is mean Li-Yorke chaotic with respect to $d_F$, but not with respect to $d_B$. This fact follows easily from the boundedness of the metric $d_F$. Indeed, for every dynamical system with a bounded metric we have that the notion of mean proximal pair coincides with that of distributionally proximal pair, but also the notion of mean asymptotic pair coincides with that of distributionally asymptotic pair (recall that a mean Li-Yorke chaotic pair is a mean proximal pair that is not a mean asymptotic pair). See Remark~\ref{rem:relation-of-D-M} for more details.
\end{rem}

\section{Distributional chaos}
\label{sec:DC}
\subsection{Mean-L-stability, mean-L-unstability and distributional chaos in topological dynamics}
\label{subsec:MLS-DC}

Let $(X,T)$ be a dynamical system with a metric $d$ on $X$.
A pair $(x,y)\in X\times X$ is called \emph{distributionally asymptotic} if for every $\eps>0$,
\[
    \dens(\{n\in\bbn\colon d(T^nx,T^ny)< \eps\})=1,
\]
and \emph{distributionally proximal} if for every $\eps>0$,
\[
    \udens(\{n\in\bbn\colon d(T^nx,T^ny)< \eps\})=1.
\]
The \emph{distributionally asymptotic relation} and the \emph{distributionally proximal relation} of $(X,T)$, denoted by $\dasym(T)$ and $\dprox(T)$, are the set of all distributionally asymptotic pairs and  distributionally proximal pairs respectively.
For any $x\in X$, the \emph{distributionally asymptotic} and the \emph{distributionally proximal} cell of $x$ are defined by
\[
    \dasym(T,x)=\{y\in X\colon (x,y)\in \dasym(T)\}
\]
and
\[
    \dprox(T,x)=\{y\in X \colon (x,y) \in \dprox(T)\},
\]
respectively.

The following three lemmas are folklore. Similar to Lemmas~\ref{lem:dyn-prox-delta} and ~\ref{lem:mean-prox-G-delta}, we have the following $G_\delta$ sets. 
\begin{lem}
    Let $(X,T)$ be a dynamical system and $\delta,\eta>0$.
    Then
    \begin{enumerate}
        \item for every $x\in X$, the distributionally proximal cell of $x$ is a $G_\delta$ subset of $X$;
        \item the distributionally proximal relation is a $G_\delta$ subset of $X\times X$;
        \item for every $x\in X$, the set
              \[
                  \bigl\{y\in X\colon \udens(\{n\in\bbn\colon d(T^nx,T^ny)> \delta \})\geq \eta\bigr\}
              \]
              is a $G_\delta$ subset of $ X$;
        \item the set
              \[
                   \bigl\{(x,y)\in X\times X\colon \udens(\{n\in\bbn\colon d(T^nx,T^ny)> \delta \})\geq \eta \bigr\}
              \]
              is a  $G_\delta$ subset of $X\times X$.
    \end{enumerate}
\end{lem}
\begin{proof}
    (1) follows from the following formula:
    \begin{align*}
        \dprox(T,x)=\bigcap_{n=1}^\infty
        \bigl\{ y\in X\colon 
        & \exists k> n \text{ s.t. }\\                
        & \card\bigl(\{1\leq i\leq k\colon  d(T^ix,T^iy)<\tfrac{1}{n}\}\bigr)> k(1-\tfrac{1}{n})\bigr\}.
             \end{align*}
    The proofs of (2), (3) and (4) are similar to that of (1).
\end{proof}

Following~\cite{F1951}, we say that a dynamical system $(X,T)$ is \emph{stable in the mean in the sense of Lyapunov} (or \emph{mean-L-stable} briefly)
if for any $\eps>0$ there exists some $\delta>0$ such that
for every $x,y\in X$ with $d(x,y)<\delta$,
\[
    \udens(\{n\in\bbn\colon d(T^nx,T^ny)\geq \eps\})<\eps.
\]

\begin{lem}\label{lem:mean-L-stable-prox}
    Let $(X,T)$ be a dynamical system.
    If $(X,T)$ is mean-L-stable, then every proximal pair $(x,y)\in X\times X$ is distributionally asymptotic.
\end{lem}
\begin{proof}
    Let $(x_0,y_0)\in X\times X$ be a proximal pair.
    As $(X,T)$ is mean-L-stable, for any $\eps>0$ there exists $\delta>0$ such that
    for every $x,y\in X$ with $d(x,y)<\delta$,
    \[
        \udens(\{n\in\bbn\colon d(T^nx,T^ny)\geq \eps\})<\eps.
    \]
    Since $(x_0,y_0)$ is proximal, there exists $k\in\bbn$ such that $d(T^kx_0,T^ky_0)<\delta$. Then
    \[
        \udens(\{n\in\bbn\colon d(T^nx_0,T^ny_0)\geq \eps \})
        =\udens(\{n\in\bbn\colon d(T^n(T^kx_0),T^n(T^ky_0))\geq \eps\})<\eps.
    \]
    By the arbitrariness of $\eps>0$, $(x_0,y_0)$ is distributionally asymptotic.
\end{proof}

We say that a dynamical system $(X,T)$ is \emph{unstable in the mean in the sense of Lyapunov} (or \emph{mean-L-unstable} briefly)
if there exists some $\delta>0$ such that, for every $x\in X$ and $\eps>0$,
there exists some $y\in X$ with $d(x,y)<\eps$ such that
\[
    \udens(\{n\in\bbn\colon d(T^nx,T^ny)>\delta\})\geq \delta.
\]

The proof of the following result is similar to that of Lemma~\ref{lem:sen-equivalent}, and we leave it to the reader.

\begin{lem}
    Let $(X,T)$ be a dynamical system with $X$ being completely metrizable.
    Then the following assertions are equivalent:
    \begin{enumerate}
        \item $(X,T)$ is mean-L-unstable;
        \item there exists $\delta>0$ such that for every $x\in X$, the set
              \[
                   \bigl\{ y\in X\colon \udens(\{n\in\bbn\colon d(T^nx,T^ny)>\delta\})\geq \delta \bigr\}
              \]
              is a dense $G_\delta$ subset of $X$;
        \item there exists $\delta>0$ such that the set
              \[
                   \bigl\{ (x,y)\in X\times X \colon \udens(\{n\in\bbn\colon d(T^nx,T^ny)> \delta\})\geq \delta \bigr\}
              \]
              is a dense $G_\delta$ subset of $X\times X$.
    \end{enumerate}
\end{lem}

Let $(X,T)$ be a dynamical system.
For a given $\delta>0$, a pair $(x,y)\in X\times X$ is called \emph{distributionally $\delta$-chaotic of type $2$}
if $(x,y)$ is distributionally proximal and
\[
    \udens(\{n\in\bbn\colon d(T^nx,T^ny)>\delta\})\geq \delta.
\]
A subset $K$ of $X$ is called \emph{distributionally scrambled of type $2$} if any two distinct points $x,y\in K$ form a distributionally $\delta$-chaotic pair of type $2$ for some $\delta=\delta(x,y)>0$.
We say that the dynamical system $(X,T)$ is \emph{distributionally chaotic of type $2$}
if there exists an uncountable distributionally scrambled subset of type $2$.

Similarly, we can define \emph{distributionally $\delta$-scrambled set of type $2$} and \emph{distributional $\delta$-chaos of type $2$}.

The proof of the following result is analogous to that of Proposition~\ref{prop:dense-delta-scrambled}.

\begin{prop}
    Let $(X,T)$ be a dynamical system with $X$ being completely metrizable.
    Then the following assertions are equivalent:
    \begin{enumerate}
        \item there exists $\delta>0$ such that for every sequence $(O_j)_j$ of nonempty open subsets of $X$, there exists a sequence $(K_j)_j$ of Cantor sets with $K_j\subset O_j$ such that $\bigcup_{j=1}^\infty K_j$ is distributionally $\delta$-scrambled of type $2$;
        \item the distributionally proximal relation of $(X,T)$ is dense in $X\times X$ and $(X,T)$ is mean-L-unstable.
    \end{enumerate}
\end{prop}

\begin{rem}\label{rem:relation-of-D-M}
  Given a dynamical system $(X,T)$, it is not hard to check that a pair $(x,y)$ is distributionally $\delta$-chaotic of type 2, for some $\delta>0$, if and only if $(x,y)$ is distributionally proximal but not distributionally asymptotic. Note also that we have the following inclusions

\[
\begin{array}{ccccccc}
\asym(T)&\subset& \masym(T)& \subset& \dasym(T)& &\\[3pt]
& &\cap & & \cap &  & \\[3pt]
& &\mprox(T)& \subset& \dprox(T)& \subset & \prox(T),
\end{array}
\]
but also the implications
\[
   \text{Equicontinuity}\Rightarrow \text{Mean Equicontinuity}\Rightarrow \text{Mean-L-Stability}.
\]
These inclusions and implications follow either trivially, or else because of the following inequality exposed in \cite[Proposition 20]{BBPW2018} and that holds for every $x,y\in X$ and $\delta>0$:
\[
    \frac{\card(\{1\leq j\leq N: d(T^j x, T^j y)\geq \delta\})}{N}\leq \frac{1}{N}\sum_{j=1}^{N}\frac{d(T^j x,T^j y)}{\delta}=\frac{1}{\delta}\frac{1}{N}\sum_{j=1}^{N}d(T^jx, T^j y).
\]
Moreover, it is well-known (see references \cite{D2014} and \cite{LTY2015}) that when the metric $d$ on $X$ is bounded then we have the equalities
\[
    \masym(T)=\dasym(T) \qquad \text{and} \qquad \mprox(T)=\dprox(T),
\]
but also the equivalence
\[
    \text{Mean Equicontinuity}\Leftrightarrow\text{Mean-L-Stability}.
\]
These equalities and equivalence, when $d$ is bounded by some value $K>0$, follow since for every $x, y\in X$ and $\delta>0$ we have that 
  \begin{align*}
        \frac{1}{N}\sum_{j=1}^{N}d(T^j x,T^j y)&\leq \frac{1}{N}\bigl (\delta \cdot \card(\{1\leq j\leq N: d(T^j x, T^j y)<\delta\}) \\
        &\qquad\quad +K \cdot \card(\{1\leq j\leq N: d(T^j x, T^j y)\geq\delta\})\bigr)\\
        &\leq \delta+K \cdot \frac{\card(\{1\leq j\leq N: d(T^j x, T^j y)\geq \delta\})}{N}.
  \end{align*}
These arguments trivially explain why the operator considered in Remark~\ref{rem:m-of-BF-BB} is mean Li-Yorke chaotic with respect to the metric $d_F$ but not with respect to the metric $d_B$.
\end{rem}

For a given $\delta>0$, a pair $(x,y)\in X\times X$ is called \emph{distributionally $\delta$-chaotic of type $1$}
if $(x,y)$ is distributionally proximal and
\[
    \udens(\{n\in\bbn\colon d(T^nx,T^ny)> \delta\})=1.
\]
A subset $K$ of $X$ is called \emph{distributionally scrambled of type $1$} if any two distinct points $x,y\in K$ form a distributionally $\delta$-chaotic pair of type $1$ for some $\delta=\delta(x,y)>0$.
We say that the dynamical system $(X,T)$ is \emph{distributionally chaotic of type $1$}
if there exists an uncountable distributionally scrambled subset of type $1$.

Similarly, we can define \emph{distributionally $\delta$-scrambled set of type $1$} and \emph{distributional $\delta$-chaos of type $1$}.

The proof of the following result is analogous to that of Proposition~\ref{prop:dense-delta-scrambled}.

\begin{prop}\label{prop:dense-DC1}
    Let $(X,T)$ be a dynamical system with $X$ being completely metrizable.
    Then the following assertions are equivalent:
    \begin{enumerate}
        \item there exists $\delta>0$ such that for every sequence $(O_j)_j$ of nonempty open subsets of $X$, there exists a sequence $(K_j)_j$ of Cantor sets with $K_j\subset O_j$ such that $\bigcup_{j=1}^\infty K_j$ is distributionally $\delta$-scrambled of type $1$;
        \item the distributionally proximal relation of $(X,T)$ is dense in $X\times X$, and there exists $\delta>0$ such that the set
              \[
                   \bigl\{(x,y)\in X\times X\colon \udens(\{n\in\bbn\colon d(T^nx,T^ny)> \delta\})=1 \bigr\}
              \] is dense in $X\times X$.
    \end{enumerate}
\end{prop}

Recall that every Fr\'echet space has a separating increasing sequence of seminorms, and that each of these seminorms can be seen as a pseudometric.
In order to apply the theory developed previously to the case of Fr\'echet spaces,
we introduce the concepts of extreme mean-L-unstability and distributional extreme chaos with respect to a continuous pseudometric.
From now on in this subsection, we fix a continuous pseudometric $\rho$ on $X$.
We say that a dynamical system $(X,T)$ is \emph{$\rho$-extremely mean-L-unstable} 
if there exists some $\delta>0$ such that, for every $x\in X$ and $\eps>0$,
there exists some $y\in X$ with $d(x,y)<\eps$ such that, for any $M>0$,
\[
    \udens(\{n\in\bbn\colon \rho(T^nx,T^ny)>M\})\geq \delta.
\]

The proof of the following result is similar to that of Lemma~\ref{lem:sen-equivalent} and we leave it to the reader.

\begin{lem}
    Let $(X, T)$ be a dynamical system with $X$ being completely metrizable.
    Then the following assertions are equivalent:
    \begin{enumerate}
        \item $(X,T)$ is $\rho$-extremely mean-L-unstable;
        \item there exists $\delta>0$ such that for every $x\in X$, the set
              \[
                   \bigl\{ y\in X\colon \udens(\{n\in\bbn\colon \rho(T^nx,T^ny)>M\})\geq \delta,\ \forall M>0 \bigr\}
              \]
              is a dense $G_\delta$ subset of $X$;
        \item there exists $\delta>0$ such that the set
              \[
                   \bigl\{ (x,y)\in X\times X \colon \udens(\{n\in\bbn\colon \rho(T^nx,T^ny)> M\})\geq \delta, \forall M>0 \bigr\}
              \]
              is a dense $G_\delta$ subset of $X\times X$.
    \end{enumerate}
\end{lem}

For a given $\delta>0$, a pair $(x,y)\in X\times X$ is called \emph{distributionally $\rho$-extremely $\delta$-chaotic of type $2$}
if $(x,y)$ is distributionally proximal and for any $M>0$
\[
    \udens(\{n\in\bbn\colon \rho(T^nx,T^ny)>M\})\geq \delta,
\]
and \emph{distributionally $\rho$-extremely chaotic of type $1$}
if $(x,y)$ is distributionally proximal and for any $M>0$
\[
    \udens(\{n\in\bbn\colon \rho(T^nx,T^ny)> M\})=1.
\]
Similarly,
we can define \emph{distributionally $\rho$-extremely $\delta$-scrambled set of type $2$} and \emph{distributional $\rho$-extreme  $\delta$-chaos of type $2$}, \emph{distributionally $\rho$-extremely scrambled set of type $1$} and \emph{distributional $\rho$-extreme chaos of type $1$}.

The proof of the following two results is analogous to that of Proposition~\ref{prop:dense-delta-scrambled}.

\begin{prop}
    Let $(X,T)$ be a dynamical system with $X$ being completely metrizable.
    Then the following assertions are equivalent:
    \begin{enumerate}
        \item there exists $\delta>0$ such that for every sequence $(O_j)_j$ of nonempty open subsets of $X$, there exists a sequence $(K_j)_j$ of Cantor sets with $K_j\subset O_j$ such that $\bigcup_{j=1}^\infty K_j$ is distributionally $\rho$-extremely $\delta$-scrambled of type $2$;
        \item the distributionally proximal relation of $(X,T)$ is dense in $X\times X$ and $(X,T)$ is $\rho$-extremely mean-L-unstable.
    \end{enumerate}
\end{prop}

\begin{prop}
    Let $(X,T)$ be a dynamical system with $X$ being completely metrizable.
    Then the following assertions are equivalent:
    \begin{enumerate}
        \item  for every sequence $(O_j)_j$ of nonempty open subsets of $X$, there exists a sequence $(K_j)_j$ of Cantor sets with $K_j\subset O_j$ such that $\bigcup_{j=1}^\infty K_j$ is distributionally $\rho$-extremely scrambled of type $1$;
        \item the distributionally proximal relation of $(X,T)$ is dense in $X\times X$, and the set
              \[
                   \bigl\{(x,y)\in X\times X\colon \udens(\{n\in\bbn\colon \rho(T^nx,T^ny)>M\})=1,\ \forall M>0 \bigr\}
              \] is dense in $X\times X$.
    \end{enumerate}
\end{prop}

\begin{rem}
    According to Lemma~\ref{lem:density-limit}, we can give  equivalent definitions of distributional chaos of type $1$ and $2$ by limits along subsets of $\bbn$.
\end{rem}

\subsection{Distributional chaos for continuous endomorphisms of completely metrizable groups}
\label{subsec:DC-group}
In \cite{BBMP2011}, \cite{BBMP2013} and \cite{BBPW2018} the authors characterized distributional chaos for linear operators on Banach and Fr\'echet spaces. In this section, we will generalize some of their results to the setting of continuous endomorphisms of completely metrizable groups.

In this subsection, we fix a completely metrizable group $G$ and a left-invariant compatible metric $d$ on $G$.
Let $T\in\eend(G)$. A point $x\in G$ is called \emph{distributionally semi-irregular of type $2$} for $T$ if
there exists a subset $A$ of $\bbn$ with $\udens(A)=1$ such that
$\lim\limits_{A\ni n\to\infty} T^n x=e$ and there exists a subset $B$ of $\bbn$ with $\udens(B)>0$ such that $\liminf\limits_{B\ni n\to\infty} d(T^n x,e)>0$,
and \emph{distributionally semi-irregular of type $1$} for $T$ if
there exists a subset $A$ of $\bbn$ with $\udens(A)=1$ such that
$\lim\limits_{A\ni n\to\infty} T^n x=e$ and there exists a subset $B$ of $\bbn$ with $\udens(B)=1$ such that $\liminf\limits_{B\ni n\to\infty} d(T^n x,e)>0$.

The proof of the following result is similar to that of Lemma~\ref{lem:g-asym-prox-left-inv}.

\begin{lem}
    Let $T\in\eend(G)$.
    \begin{enumerate}
        \item For every $x\in G$, $\dasym(T,x) = x \cdot \dasym(T,e)$,
              $\dprox(T,x)=x\cdot \dprox(T,e)$;
        \item  $\dasym(T,e)$ is dense in $G$ if and only if $\dasym(T)$ is dense in $G\times G$;
        \item If $\dasym(T,e)$ is residual in $G$, then $\dasym(T,e)=G$;
        \item $\dprox(T,e)$ is dense in $G$ if and only if $\dprox(T)$ is dense in $ G\times G$;
        \item For $i=1,2$ and every $x,y\in G$, $(x,y)$ is distributionally chaotic of type $i$ if and only if $x^{-1}y$ is a distributionally semi-irregular point of type $i$.
        \item For $i=1,2$, if a subset $K$  of $G$ is distributionally scrambled of type $i$, then for any $x\in G$, $xK$ is also  distributionally scrambled of type $i$.  In particular, for any $x\in K$, $x^{-1}K\setminus \{e\}$ consists of distributionally semi-irregular points of type $i$.
        \item For $i=1,2$ and a given $\delta>0$, if a subset $K$  of $G$ is distributionally $\delta$-scrambled of type $i$, then for any  $x\in G$, $xK$ is also distributionally $\delta$-scrambled of type $i$ .
    \end{enumerate}
\end{lem}

\begin{rem}
    Let $T\in \eend(G)$. If $(x,e)$ is distributionally proximal, then  there exists a subset $A$ of $\bbn$ with $\udens(A)=1$ such that
    $\lim\limits_{A\ni n\to\infty} T^n x=e$. Let
    \[
        G_0=\Bigl\{ y\in G\colon \lim\limits_{A\ni n\to\infty} T^n y=e\Bigr\}.
    \]
    It is clear that $G_0\subset \dprox(T,e)$.
    By Lemma~\ref{lem:subseq-asym-subgroup}, $G_0$ is a subgroup of $G$.
\end{rem}

The proof of the following result is analogous that of Theorems~\ref{thm:dich-eq-sen} and~\ref{thm:dich-mean-eq-mean-sen}.
\begin{thm}\label{thm:dich-mean-L-stable-sens-mean}
    Let $T\in\eend(G)$.
    Then $T$ is either mean-L-stable or  mean-L-unstable.
\end{thm}

Combining Lemma~\ref{lem:mean-L-stable-prox} and Theorem ~\ref{thm:dich-mean-L-stable-sens-mean}, we have the following.
\begin{coro}\label{coro:dist-scambled-sen}
    Let $T\in\eend(G)$. If there exists a distributionally chaotic pair of type $2$,
    then $T$ is mean-L-unstable.
\end{coro}

The proof of the following result is analogous that of Theorem~\ref{prop:g-mean-sen-eq-cond}.
\begin{prop} \label{prop:g-sen-in-mean-eq-cond}
    Let $T\in \eend(G)$. Then the following assertions are equivalent:
    \begin{enumerate}
        \item $T$ is mean-L-unstable;
        \item there exists $\delta>0$ such that for any $x\in G$ and $\eps>0$
              there exists $y\in X$ with $d(x,y)<\eps$ and
              \[
                  \sup_{n\in\bbn} \frac{1}{n}\card(\{1\leq i\leq n\colon d(T^ix,T^iy)>\delta\}) \geq \delta;
              \]
        \item there exists $\delta>0$, a sequence $(y_k)_k$ in $G$ and a sequence $(N_k)_k$ in $\bbn$ such that
              $\lim_{k\to\infty} y_k=e$ and
              \[
                  \inf_{k\in\bbn} \frac{1}{N_k}\card(\{1\leq i\leq N_k\colon d(T^iy_k,e)>\delta\}) \geq \delta;
              \]
    \end{enumerate}
\end{prop}

Combining Theorem~\ref{thm:dich-mean-L-stable-sens-mean} and Proposition~\ref{prop:g-sen-in-mean-eq-cond}, we have the following.
\begin{coro}
    Let $T\in \eend(G)$. Then $T$ is mean-L-stable if and only if for any $\eps>0$ there exists some $\delta>0$ such that
    for every $x,y\in G$ with $d(x,y)<\delta$,
    \[
        \sup_{n\in\bbn} \frac{1}{n}\card(\{1\leq i\leq n\colon d(T^ix,T^iy)\geq \eps\}) <\eps.
    \]
\end{coro}

Now we give a characterization of the existence of a dense set of distributionally semi-irregular points for continuous endomorphisms of completely metrizable groups.
Since the proof is similar to that of Theorem~\ref{thm:equi-of-dense-LY}, we leave it to the reader.

\begin{thm}\label{thm:equi-of-dense-DC2}
    Let $T\in\eend(G)$.
    Then the following assertions are equivalent:
    \begin{enumerate}
        \item $T$ admits a dense set of distributionally semi-irregular points of type $2$;
        \item $T$ admits a residual set of distributionally semi-irregular points of type $2$;
        \item there exists $\delta>0$ such that for every sequence $(O_j)_j$ of nonempty open subsets of $G$, there exists a sequence $(K_j)_j$ of Cantor sets with $K_j\subset O_j$ such that $\bigcup_{j=1}^\infty K_j$ is distributionally $\delta$-scrambled of type $2$;
        \item the distributionally proximal relation of $T$ is dense in $G\times G$ and $T$ is mean-L-unstable;
        \item the distributionally proximal cell of $e$ is dense in $G$ and there exists $x\in G$ and $\delta>0$ such that
              \[
                  \udens(\{n\in\bbn\colon d(T^nx,e)> \delta )>0.
              \]
    \end{enumerate}
\end{thm}

According to Theorem~\ref{thm:equi-of-dense-DC2}\,(5), we have the following.
\begin{coro}
    Let $T\in\eend(G)$. If the distributionally proximal cell of $e$ is dense in $G$ and $T$
    has a non-trivial periodic point, then it has a dense set of distributionally semi-irregular points of type $2$.
\end{coro}

By applying Theorem~\ref{thm:equi-of-dense-DC2},
we also can characterize distributional chaos of type $2$
for continuous endomorphisms of completely metrizable groups.
\begin{thm}
    Let $T\in\eend(G)$.
    Then the following assertions are equivalent:
    \begin{enumerate}
        \item $T$ admits a distributionally semi-irregular point of type $2$;
        \item $T$ admits a distributionally chaotic pair of type $2$;
        \item $T$ is distributionally chaotic of type $2$;
        \item the restriction of $T$ to some closed $T$-invariant subgroup $\widetilde{G}$ has a dense set of distributionally semi-irregular points of type $2$.
    \end{enumerate}
\end{thm}

Similar to Proposition~\ref{prop:residul-scrambled-set}, we have the following result.

\begin{prop}
    Let $T\in\eend(G)$.
    Then the following assertions are equivalent:
    \begin{enumerate}
        \item $T$ admits a residual distributionally scrambled set of type $2$;
        \item every point in $G\setminus \{e\}$ is distributionally semi-irregular of type $2$ for $T$;
        \item $X$ is a distributionally scrambled set of type $2$ for $T$.
    \end{enumerate}
\end{prop}

\begin{exam}\label{exam:SZ-DC-LY-chaos}
Let $\bbs(\bbz),\sigma, T, A$ as in Example~\ref{exam:SZ-LY-chaos}. 
As the distributionally proximal relation contains the asymptotic relation 
and $\sigma$ is a fixed point for $T$, 
according to Theorem~\ref{thm:equi-of-dense-DC2} $(\bbs(\bbz), T)$ is densely distributionally chaotic of type $2$ for any left-invariant compatible metric on $\bbs(\bbz)$.

Let $d$ be the precise left-invariant compatible metric as in Example~\ref{exam:SZ-LY-chaos}. It is easy to show that 
there exists $\delta>0$ such that the set
\[
 \bigl\{(x,y)\in \bbs(\bbz)\times \bbs(\bbz)\colon \udens(\{n\in\bbn\colon d(T^nx,T^ny)> \delta\})=1 \bigr\}
\] 
is dense in $\bbs(\bbz)\times \bbs(\bbz)$.
Then, by Proposition~\ref{prop:dense-DC1}, $(\bbs(\bbz), T)$ is densely distributionally chaotic of type $1$ with respect to the metric $d$.
\end{exam}

\begin{exam}\label{exam:HR-DC-LY-chaos}
Let $H_+(\bbr),\sigma, T, A$ as in Example~\ref{exam:HR-LY-chaos}.
Similar to Example~\ref{exam:SZ-DC-LY-chaos}, one has $(H_+(\bbr), T)$ is densely distributionally chaotic of type $2$ for any left-invariant compatible metric on $H_+(\bbr)$, and densely distributionally chaotic of type $1$ with respect to the precise metric compatible $d$ as in Example~\ref{exam:HR-LY-chaos}.
\end{exam}

\subsection{Distributional extreme chaos for continuous endomorphisms of completely metrizable groups}
\label{subsec:DeC-group}
With the same notation as in the previous subsection, we fix a completely metrizable group $G$ and a left-invariant compatible metric $d$ on $G$.
In order to apply the results obtained in this section to the Fréchet space setting, 
in this subsection we fix a left-invariant continuous pseudometric $\rho$ on $G$.

Let $T\in\eend(G)$.
A point $x\in G$ is called \emph{distributionally irregular of type $2$} if there exists a subset $A$ of $\bbn$ with $\udens(A)=1$ such that
$\lim\limits_{A\ni n\to\infty} T^nx=e$, and there exists a subset $B$ of $\bbn$ with $\udens(B)>0$ such that
$\lim\limits_{B\ni n\to\infty} \rho(T^nx,e)=\infty$,
and \emph{distributionally irregular of type $1$} for $T$ if
there exists a subset $A$ of $\bbn$ with $\udens(A)=1$ such that
$\lim\limits_{A\ni n\to\infty} T^nx=e$, and there exists a subset $B$ of $\bbn$ with $\udens(B)=1$ such that
$\lim\limits_{B\ni n\to\infty} \rho(T^nx,e)=\infty$.

The proofs of the following three results are similar to those of Lemma~\ref{lem:g-asym-prox-left-inv} and Propositions~\ref{prop:g-mean-sen-eq-cond} and~\ref{prop:irregular-ext-sen} respectively.

\begin{lem}\label{lem:g-dis-ext-scrambled-irr}
    Let $T\in\eend(G)$.
    \begin{enumerate}
        \item For $i=1,2$, and every $x,y\in G$, $(x,y)$ is distributionally extremely chaotic of type $i$ if and only if $x^{-1}y$ is a distributionally irregular point of type $i$.
        \item For $i=1,2$, if $K\subset G$ is a distributionally extremely scrambled subset of type $i$,
              then for any $x\in G$, $xK$ is also a distributionally extremely scrambled set of type $i$. In particular, for any $x\in K$, $x^{-1}K\setminus \{e\}$ consists of distributionally irregular points of type $i$.
    \end{enumerate}
\end{lem}

\begin{prop}\label{prop:g-dis-ext-sen-eq-cond}
    Let $T\in \eend(G)$. Then the following assertions are equivalent:
    \begin{enumerate}
        \item $T$ is  extremely mean-L-unstable;
        \item there exists some $\delta>0$ such that for every $x\in G$ and $\eps,M>0$
              there exists $y\in G$ with $d(x,y)<\eps$ and
              \[
                  \sup_{n\in\bbn} \frac{1}{n}\card(\{1\leq i\leq n\colon \rho(T^ix,T^iy)>M\})\geq \delta;
              \]
        \item there exists a sequence $(y_k)_k$ in $G$ and a sequence $(N_k)_k$ in $\bbn$ such that
              $\lim_{k\to\infty} y_k=e$ and
              \[
                  \inf_{k\in\bbn} \frac{1}{N_k} \card(\{1\leq i\leq N_k \colon  \rho(T^i y_k,e)> k\})>0;
              \]
    \end{enumerate}
\end{prop}

\begin{prop}\label{prop:dis-irr-ext-sen}
    Let $T\in \eend(G)$. If there exists a distributionally irregular point of type $2$, then $T$ is extremely mean-L-unstable.
\end{prop}

Now we give a characterization of the existence of a dense set of distributionally irregular points of type $2$ for continuous endomorphisms of completely metrizable groups. Since the proof is similar to that of Theorem~\ref{thm:equi-of-dense-irr-points}, we leave it to the reader.

\begin{thm}\label{thm:g-dense-dis-irr-points-2}
    Let $T\in\eend(G)$.
    Then the following assertions are equivalent:
    \begin{enumerate}
        \item $T$ admits a dense set of distributionally irregular points of type $2$;
        \item $T$ admits a residual set of distributionally irregular points of type $2$;
        \item there exists $\delta>0$ such that for every sequence $(O_j)_j$ of nonempty open subsets of $G$, there exists a sequence $(K_j)_j$ of Cantor sets with $K_j\subset O_j$ such that $\bigcup_{j=1}^\infty K_j$ is distributionally extremely $\delta$-scrambled of type $2$;
        \item the distributionally proximal relation of $T$ is dense in $G\times G$ and $T$ is extremely mean-L-unstable.
    \end{enumerate}
\end{thm}

By applying Theorem~\ref{thm:g-dense-dis-irr-points-2},
we also can characterize distributional extreme chaos of type $2$
for continuous endomorphisms of completely metrizable groups.
\begin{thm}
    Let $T\in\eend(G)$.
    Then the following assertions are equivalent:
    \begin{enumerate}
        \item $T$ admits a distributionally irregular point of type $2$;
        \item $T$ admits a distributionally extremely chaotic pair of type $2$;
        \item $T$ is distributionally extremely chaotic of type $2$;
        \item the restriction of $T$ to some closed $T$-invariant subgroup $\widetilde{G}$ has a dense set of distributionally irregular points of type $2$.
    \end{enumerate}
\end{thm}

Let $T\in\eend(G)$.
A point $x\in G$ is called \emph{distributionally unbounded} if
there exists a subset $B$ of $\bbn$ with $\udens(B)=1$ such that
$\lim\limits_{B\ni n\to\infty}\rho(T^nx,e)=\infty$.
Then a point $x\in G$ is distributionally irregular of type $1$
if and only if
$(x,e)$ is distributionally proximal and $x$ is distributionally unbounded.
For each $n\in X$, let
\[
    G_n=\bigl\{x\in G\colon \exists k>n\text{ s.t. }
    \card(\{1\leq i\leq k\colon \rho(T^i x,e)>n\})> k(1-\tfrac{1}{n})\bigr\}.
\]
It is clear that $G_n$ is open and $\bigcap_{n=1}^\infty G_n$ is the set of all distributionally unbounded points.

Similar to Theorem~\ref{thm:g-dense-dis-irr-points-2}, we have the following
characterization for the existence of a dense set of distributionally irregular points of type $1$ for continuous endomorphisms of completely metrizable groups.

\begin{thm}\label{thm:g-dense-dc1}
    Let $T\in\eend(G)$.
    Then the following assertions are equivalent:
    \begin{enumerate}
        \item $T$ admits a dense set of distributionally irregular points of type $1$;
        \item $T$ admits a residual set of distributionally irregular points of type $1$;
        \item for every sequence $(O_j)_j$ of nonempty open subsets of $G$, there exists a sequence $(K_j)_j$ of Cantor sets with $K_j\subset O_j$ such that $\bigcup_{j=1}^\infty K_j$ is distributionally extremely scrambled of type $1$;
        \item the distributionally proximal cell of $e$ is dense in $G$ and $T$ admits a dense set of distributionally unbounded points.
    \end{enumerate}
\end{thm}

\subsection{Distributional chaos for continuous linear operators on Fr{\'e}chet spaces}
\label{subsec:DC-Frechet}
Distributional chaos of type 1 for linear operators on Fr\'echet spaces has been characterized in \cite{BBMP2013}, and four types of distributional chaos for linear operators on Banach spaces are investigated in \cite{BBPW2018}. It should be noticed that in \cite{YLW21} Yin et al. studied distributional chaos for operators on Fr\'echet spaces with respect to the bounded Fr\'echet-space metric.
In this subsection using results from Subsections~\ref{subsec:DC-group} and ~\ref{subsec:DeC-group} we will 
consider distributional chaos of type 1 and 2 for linear operators on Fr\'echet spaces. With this we reprove some of the results from \cite{BBMP2013}, and we also generalize some results from \cite{BBPW2018} in the Fréchet space setting. 

In this subsection, unless otherwise specified, $X$ denotes an arbitrary infinite-dimensional Fr{\'e}chet space.
The separating increasing sequence $(p_j)_j$ of seminorms on $X$ and the metric $d$ on $X$ are as in Subsection~\ref{subsec:LYC-Frechet}.

Let $T\in L(X)$, $m\in\bbn$ and $i=1,2$.
A vector $x\in X$ is called \emph{distributionally $m$-irregular of type $i$} if it is distributionally irregular of type $i$ with respect to the seminorm $p_m$.
A pair $(x,y)\in X\times X$ is called \emph{distributionally $m$-extremely $\delta$-chaotic of type $i$} if it is distributionally extremely $\delta$-chaotic of type $i$ with respect to the seminorm $p_m$.
Similarly, we can define \emph{distributional  $m$-extreme  sensitivity}, \emph{distributionally $m$-extremely $\delta$-scrambled set of type $i$} and \emph{distributional $m$-extreme chaos of type $i$}.

\begin{prop}\label{prop:f-ext-sen-in-mean}
    Let $T\in L(X)$. Then $T$ is  mean-L-unstable if and only if it is $m$-extremely mean-L-unstable for some $m\in\bbn$.
\end{prop}
\begin{proof}
    If $T$ is $m$-extremely mean-L-unstable for some $m\in\bbn$,
    as
 \[
     d(x,y)\geq \frac{1}{2^m} \min \{p_m(x,y), 1\}
 \]
 for all $x,y\in X$, then $T$ is mean-L-unstable.
    Now assume that $T$ is  mean-L-unstable. By Proposition~\ref{prop:g-sen-in-mean-eq-cond}\,(3), there exists $\delta>0$, a sequence $(y_k)_k$ in $X$ and a sequence $(N_k)_k$ in $\bbn$ such that $\lim_{k\to\infty} y_k=\vecz$ and
    \[
        \inf_{k\in\bbn} \frac{1}{N_k}\card(\{1\leq i\leq N_k\colon d(T^iy_k,\vecz)>\delta\}) \geq \delta.
    \]
    Since $d(x,\vecz)<p_m(x)+\frac{1}{2^m}$ for all $x\in X$ and $m\in\bbn$, there exists $\delta^\prime >0$ and $m\in\bbn$ such that
    \[
        \inf_{k\in\bbn} \frac{1}{N_k}\card(\{1\leq i\leq N_k\colon p_m(T^iy_k)>\delta^\prime\}) \geq \delta.
    \]
    Without loss of generality, assume that $\lim_{k\to\infty} k\cdot y_k=\vecz$, otherwise pick a subsequence of $(y_k)_k$, still denoted by $(y_k)_k$, that converges fast enough to $\vecz$ and then satisfies the requirement. 
    For each $k\in\bbn$, let $x_k=\frac{k}{\delta'} \cdot y_k$. 
    Then $\lim_{k\to\infty} x_k=\vecz$ and
    \[
        \inf_{k\in\bbn} \frac{1}{N_k}\card(\{1\leq i\leq N_k\colon p_m(T^ix_k)>k\}) \geq \delta.
    \]
    Then by Proposition~\ref{prop:g-dis-ext-sen-eq-cond}\,(3), $T$ is $m$-extremely mean-L-unstable.
\end{proof}

Combining Theorems~\ref{thm:equi-of-dense-DC2} and~\ref{thm:g-dense-dis-irr-points-2} and Proposition~\ref{prop:f-ext-sen-in-mean},
we have the following characterization for the existence of a dense set of distributionally semi-irregular vectors of type $2$ for continuous linear operators on Fr\'echet spaces.

\begin{thm}\label{thm:f-equi-of-dense-DC2}
    Let $T\in L(X)$.
    Then the following assertions are equivalent:
    \begin{enumerate}
        \item $T$ admits a dense set of distributionally semi-irregular vectors of type $2$;
        \item $T$ admits a residual set of distributionally $m$-irregular vectors of type $2$ for some $m\in\bbn$;
        \item there exists $m\in\bbn$ and $\delta>0$ such that for every sequence $(O_j)_j$ of nonempty open subsets of $X$, there exists a sequence $(K_j)_j$ of Cantor sets with $K_j\subset O_j$ such that $\bigcup_{j=1}^\infty K_j$ is distributionally $m$-extremely $\delta$-scrambled of type $2$;
        \item the distributionally proximal relation of $T$ is dense in $X\times X$ and $T$ is mean-L-unstable;
        \item the distributionally proximal cell of $\vecz$ is dense in $X$ and there exists $x\in X$ and $\delta>0$ such that
              \[
                  \udens(\{n\in\bbn\colon d(T^nx,\vecz)>\delta \})>0.
              \]
    \end{enumerate}
\end{thm}

By Theorem~\ref{thm:f-equi-of-dense-DC2}\,(5), we have the following.
\begin{coro}
    Let $T\in L(X)$. Assume that the distributionally proximal cell of $\vecz$ is dense. If $T$ has a non-trivial periodic point or an eigenvalue $\lambda$ with $|\lambda|\geq 1$, then $T$ admits a residual set of distributionally $m$-irregular vectors of type $2$ for some $m\in\bbn$.
\end{coro}

By applying Theorem~\ref{thm:f-equi-of-dense-DC2},
we give a characterization of distributional chaos of type $2$ for continuous linear operators on Fr\'echet spaces.

\begin{thm}\label{thm:f-equi-of-DC2}
    Let $T\in L(X)$.
    Then the following assertions are equivalent:
    \begin{enumerate}
        \item $T$ admits a distributionally semi-irregular vector of type $2$;
        \item $T$ admits a distributionally $m$-irregular vector of type $2$ for some $m\in\bbn$;
        \item $T$ is distributionally $m$-extremely $\delta$-chaotic of type $2$ for some $m\in\bbn$ and $\delta>0$;
        \item the restriction of $T$ to some closed $T$-invariant subspace $\widetilde{X}$ has a dense set of
              distributionally $m$-irregular vectors of type $2$ for some $m\in\bbn$.
    \end{enumerate}
\end{thm}

Let $T\in L(X)$ and $m\in\bbn$.
We say that $T$ is \emph{distributionally sensitive} if there exists some $\delta>0$ such that, for every $x\in X$ and $\eps>0$, 
there exists some $y\in X$ with $ d(x,y)<\eps$ such that
\[
    \udens(\{n\in\bbn \colon d(T^nx,T^ny)>\delta\})=1,
\]
and \emph{distributionally $m$-extremely sensitive} if for every $x\in X$ and $\eps,M>0$
there exists some $y\in X$ with $ d(x,y)<\eps$ such that
\[
    \udens(\{n\in\bbn \colon p_m(T^nx-T^ny)>M\})=1.
\]
A vector $x\in X$ is called \emph{distributionally $m$-unbounded} if  there exists a subset $A$ of $\bbn$ with $\udens(A)=1$ such that $\lim\limits_{A\ni n\to\infty}p_m(T^n x)=\infty$.

The following result is essentially contained in 
\cite[Propositions~7 and~8]{BBMP2013}.
One can also provide a proof using the ideas from Propositions~\ref{prop:g-dis-ext-sen-eq-cond} and~\ref{prop:f-ext-sen-in-mean}.

\begin{prop} \label{prop:f-distr-sen}
    Let $T\in L(X)$. Then the following assertions are equivalent:
    \begin{enumerate}
        \item  $T$ is distributionally sensitive;
        \item there exists $\delta>0$, a sequence $(y_k)_k$ in $X$ and
              an increasing sequence $(N_k)_k$ in $\bbn$ such that $\lim_{k\to\infty}y_k=\vecz$ and for any $k\in\bbn$
              \[
                  \card(\{1\leq i\leq N_k\colon d(T^i y_k,\vecz)>\delta\})\geq N_k(1-\tfrac{1}{k});
              \]
        \item there exists $m\in\bbn$, $\delta>0$, a sequence $(y_k)_k$ in $X$ and
              an increasing sequence $(N_k)_k$ in $\bbn$ such that $\lim_{k\to\infty}y_k=\vecz$ and for any $k\in\bbn$
              \[
                  \card(\{1\leq i\leq N_k\colon p_m(T^i y_k)>\delta\})\geq N_k(1-\tfrac{1}{k});
              \]
        \item $T$ is distributionally $m$-extremely sensitive for some $m\in\bbn$;
        \item there exists $m\in\bbn$ and $y\in X$ with distributionally $m$-unbounded orbit;
        \item there exists $m\in\bbn$ such that the set of all points $y\in X$ with distributionally $m$-unbounded orbits is residual in $X$.
    \end{enumerate}
\end{prop}

According to Proposition~\ref{prop:f-distr-sen}, following the idea of the proof of Proposition~\ref{prop:irregular-ext-sen}, we have the following result.

\begin{prop}\label{prop:DC1-dis-sen}
    Let $T\in L(X)$. If $T$ admits a distributionally chaotic pair of type $1$ then $T$ is distributionally sensitive.
\end{prop}

Note that several sufficient conditions for dense distributional chaos for linear operators on separable Fr\'echet spaces are obtained in 
\cite[Theorems 15 and 16]{BBMP2013}.
Combining Theorem~\ref{thm:g-dense-dc1} and Propositions~\ref{prop:f-distr-sen} and \ref{prop:DC1-dis-sen},
we give a characterization of the existence of a dense set of distributionally semi-irregular points of type $1$ for continuous linear operators on Fr\'echet spaces.

\begin{thm}\label{thm:f-dense-dc1}
    Let $T\in L(X)$. Then the following assertions are equivalent:
    \begin{enumerate}
        \item $T$ admits a dense set of distributionally semi-irregular vectors of type $1$;
        \item $T$ admits a residual set of distributionally $m$-irregular vectors of type $1$ for some $m\in\bbn$;
        \item there exists $m\in\bbn$ such that for every sequence $(O_j)_j$ of nonempty open subsets of $X$, there exists a sequence $(K_j)_j$ of Cantor sets with $K_j\subset O_j$ such that $\bigcup_{j=1}^\infty K_j$ is distributionally $m$-extremely scrambled of type $1$;
        \item the distributionally proximal relation of $T$ is dense in $ X\times X$ and $T$ is distributionally sensitive.
    \end{enumerate}
\end{thm}

By applying Theorem~\ref{thm:f-dense-dc1}, we also can characterize distributional chaos of type $1$ for continuous linear operators on Fr\'echet spaces.  It should be noticed that the equivalences of (1)-(4) are already proved in \cite[Theorem 12]{BBMP2013}.

\begin{thm} \label{thm:f-dc1}
    Let $T\in L(X)$. Then the following assertions are equivalent:
    \begin{enumerate}
        \item $T$ admits a distributionally semi-irregular vector of type $1$;
        \item $T$ admits a distributionally chaotic pair of type $1$;
        \item $T$ admits a distributionally irregular vector of type $1$;
        \item $T$ is distributionally $m$-extremely chaotic of type $1$ for some $m\in\bbn$;
        \item the restriction of $T$ to some closed $T$-invariant subspace $X$ has a dense set of distributionally $m$-irregular vectors of type $1$ for some $m\in\bbn$.
    \end{enumerate}
\end{thm}

In the way of proving Theorem~\ref{thm:f-dc1} one can show the following:

\begin{prop}
    Let $T\in L(X)$. Then the set of distributionally $m$-irregular vectors of type $1$ for some $m\in\bbn$ is dense in the set of distributionally semi-irregular vectors of type $1$.
\end{prop}

Following \cite{BBMP2013}, we say that $T\in L(X)$ satisfies the \emph{distributional chaos criterion}
if there exist sequences $(x_k)_k$ and $(y_k)_k$ in $X$ such that
\begin{enumerate}
    \item there exists $A\subset\bbn$ with $\udens(A)=1$ such that $\lim\limits_{A\ni n\to\infty} T^n x_k=\vecz$ for all $k\in\bbn$;
    \item  each $y_k$ is in $\overline{\sspan(\{x_k\colon k\in\bbn\})}$, $\lim_{k\to\infty}y_k=\vecz$ and there exists $\delta>0$ and an increasing sequence $(N_k)_k$ in $\bbn$ such that for every $k\in\bbn$,
          \[
              \card(\{1\leq i\leq N_k\colon d(T^iy_k,\vecz)>\delta\})\geq N_k(1-\tfrac{1}{k}).
          \]
\end{enumerate}

Using the ideas developed in this subsection, we can give a new proof for the following result.
Since the proof is similar to that of Proposition~\ref{prop:T-chaos-criterion}, we leave it to the reader. 

\begin{prop}[{\cite[Theorem 12]{BBMP2013}}]
    Let $T\in L(X)$.
    Then $T$ is distributionally chaotic of type $1$ if and only if it satisfies the distributional chaos criterion.
\end{prop}

For $i=1,2$, a vector subspace $Y$ of $X$ is called a distributionally $m$-irregular manifold of type $i$ for $T$ if every vector $y\in Y$ is distributionally $m$-irregular of type $i$ for $T$.

Observe that for a given $\delta\in(0,1]$, 
any subset  of $\bbn$ with upper density at least $\delta$ can be partitioned in two disjoint subsets such that each of them has upper density at least  $\delta$.
Integrating this observation and the idea of the proof of Theorem~\ref{thm:dense-irregular-manifold}, we have the following two results. Note that a slightly weaker version of Proposition~\ref{prop:f-dense-dis-irr-manifold-1} was proved in \cite[Theorem 15]{BBMP2013}. 

\begin{prop} \label{prop:f-dense-dis-irr-manifold-2}
    Let $X$ be a separable Fr\'echet space and $T\in L(X)$. If the distributionally asymptotic cell of $\vecz$ is dense in $X$, then either the distributionally asymptotic cell of $\vecz$ is $X$ or it has a dense distributionally $m$-irregular manifold of type $2$ for some $m\in\bbn$.
\end{prop}

\begin{prop}\label{prop:f-dense-dis-irr-manifold-1}
    Let $X$ be a separable Fr\'echet space and  $T\in L(X)$. If the distributionally asymptotic cell of $\vecz$ is dense in $X$, then $T$ is distributionally sensitive if and only if it has a dense distributionally $m$-irregular manifold of type $1$ for some $m\in\bbn$.
\end{prop}

\subsection{Distributional chaos for continuous linear operators on Banach spaces}
\label{subsec:DC-Banach}
This section aims to slightly generalize the characterizations of distributional chaos for linear operators on Banach spaces given in \cite{BBMP2013}
and \cite{BBPW2018} via our new rather general topological dynamics approach.

Combining \cite[Proposition~8]{BBMP2013} and Propositions~\ref{prop:g-sen-in-mean-eq-cond} and~\ref{prop:f-distr-sen}, we have the following important result.
\begin{thm}\label{thm:b-mean-l-stable}
    Let $(X,\lVert\cdot\rVert)$ be a Banach space and $T\in L(X)$.
    Then $T$ is mean-L-unstable if and only if it is distributionally sensitive.
\end{thm}

Combining Theorems~\ref{thm:f-equi-of-dense-DC2}, \ref{thm:f-dense-dc1} and~\ref{thm:b-mean-l-stable}, we give a characterization of the existence of a dense
set of distributionally irregular vectors for continuous linear operators on Banach spaces.
\begin{thm}\label{thm:b-dense-dc1}
    Let $(X,\lVert\cdot\rVert)$ be a Banach space and $T\in L(X)$. Then the following assertions are equivalent:
    \begin{enumerate}
        \item $T$ admits a dense set of distributionally semi-irregular vectors of type $2$;
        \item $T$ admits a residual set of distributionally irregular vectors of type $1$;
        \item for every sequence $(O_j)_j$ of nonempty open subsets of $X$, there exists a sequence $(K_j)_j$ of Cantor sets with $K_j\subset O_j$ such that $\bigcup_{j=1}^\infty K_j$ is distributionally extremely scrambled of type $1$;
        \item the distributionally proximal relation of $T$ is dense in $ X\times X$ and $T$ is mean-L-unstable;
        \item the distributionally proximal relation of $T$ is dense in $ X\times X$ and there exists $x\in X$ and $\delta>0$ such that
              \[
                  \udens(\{n\in\bbn\colon \lVert T^n x\rVert >\delta\})>0.
              \]
    \end{enumerate}
\end{thm}

By Theorem~\ref{thm:b-dense-dc1} (5), we have the following consequence.
\begin{coro}
    Let $(X,\lVert\cdot\rVert)$ be a Banach space and $T\in L(X)$.
    Assume that the distributionally proximal cell of $\vecz$ is dense. If $T$ has a non-trivial periodic point or an eigenvalue $\lambda$ with $|\lambda|\geq 1$, then it is distributionally extremely chaotic of type $1$.
\end{coro}

By applying Theorem~\ref{thm:b-dense-dc1}, we also can characterize distributional chaos for continuous linear operators on Banach spaces, which slightly strengths \cite[Theorem~12]{BBMP2013} and \cite[Theorem~2]{BBPW2018}.

\begin{thm}\label{thm:b-dc1}
    Let $(X,\lVert\cdot\rVert)$ be a Banach space and $T\in L(X)$. Then the following assertions are equivalent:
    \begin{enumerate}
        \item $T$ admits a distributionally semi-irregular vector of type $2$;
        \item $T$ admits a distributionally chaotic pair of type $2$;
        \item $T$ admits distributionally irregular vector of type $1$;
        \item $T$ is distributionally extremely chaotic of type $1$;
        \item the restriction of $T$ to some closed $T$-invariant subspace $\widetilde{X}$ has a dense set of distributionally irregular vectors of type $1$.
    \end{enumerate}
\end{thm}

By the proof of Theorem~\ref{thm:b-dc1}, we also can obtain the following:

\begin{prop}
    Let $(X,\lVert\cdot\rVert)$ be a Banach space and $T\in L(X)$.  Then the set of distributionally irregular vectors of type $1$ is dense in the set of distributionally semi-irregular vectors of type $2$.
\end{prop}

Combining Theorem~\ref{thm:b-mean-l-stable} and Propositions~\ref{prop:f-dense-dis-irr-manifold-2} and~\ref{prop:f-dense-dis-irr-manifold-1}, we have the following result, which generalizes \cite[Theorem 33]{BBPW2018}.
\begin{prop}\label{prop:B-dense-dis-irr-manifold}
    Let $(X,\lVert\cdot\rVert)$ be a separable Banach space and $T\in L(X)$.
    If the distributionally asymptotic cell of $\vecz$ is dense in $X$, then either the distributionally asymptotic cell of $\vecz$ is $X$ or $T$ has a dense distributionally irregular manifold of type $1$.
\end{prop}

\medskip 
\noindent
\textbf{Acknowledgments:}
\thanks{J. Li (corresponding author) was supported in part by NFS of China (Grant Nos.~12222110 and 12171298).
The authors thank the anonymous referee whose careful comments have substantially improved this paper.
}

\end{document}